\documentclass[12pt,leqno,a4paper]{amsart}
\usepackage{amssymb}
\usepackage{enumitem}

\newcommand{\FF}{{\mathbb{F}}}
\newcommand{\ZZ}{\mathbb{Z}}

\newcommand{\bG}{{\mathbf{G}}}

\newcommand{\cE}{{\mathcal{E}}}
\newcommand{\cF}{{\mathcal{F}}}
\newcommand{\cG}{{\mathcal{G}}}
\newcommand{\cO}{{\mathcal{O}}}

\newcommand{\fS}{{\mathfrak{S}}}

\newcommand{\End}{{\operatorname{End}}}
\newcommand{\IBr}{{\operatorname{IBr}}}
\newcommand{\Irr}{{\operatorname{Irr}}}

\newcommand{\Sp}{{\operatorname{Sp}}}
\newcommand{\SO}{{\operatorname{SO}}}
\newcommand{\tPsi}{\tilde\Psi}

\newcommand{\tw}[1]{{}^#1\!}
\newcommand{\Ph}[1]{\Phi_{#1}}
\newcommand{\hlf}{\frac{1}{2}}
\newcommand{\thrd}{\frac{1}{3}}
\newcommand{\sxt}{\frac{1}{6}}
\newcommand{\Chevie}{{\sf Chevie}{}}
\newcommand\vr{{\vrule width16pt height3pt depth-2pt}\,}
\newcommand{\pl}{{\!+\!}}
\newcommand{\mn}{{\!-\!}}

\let\vhi=\varphi

\let\vareps=\varepsilon

\newtheorem{thm}{Theorem}[section]

\newtheorem{prop}[thm]{Proposition}

\theoremstyle{remark}
\newtheorem{rem}[thm]{Remark}

\begin{document}

\title[Decomposition matrices for exceptional groups]{Decomposition matrices\\
       for exceptional groups at $d=4$}

\date{\today}

\author{Olivier Dudas}
\address{Universit\'e Paris Diderot, UFR de Math\'ematiques,
B\^atiment Sophie Germain, 5 rue Thomas Mann, 75205 Paris CEDEX 13, France.}
\email{olivier.dudas@imj-prg.fr}

\author{Gunter Malle}
\address{FB Mathematik, TU Kaiserslautern, Postfach 3049,
         67653 Kaisers\-lautern, Germany.}
\email{malle@mathematik.uni-kl.de}

\thanks{The second author gratefully acknowledges financial support by ERC
  Advanced Grant 291512.}



\begin{abstract}
We determine the decomposition matrices of unipotent $\ell$-blocks of defect
$\Phi_4^2$ for exceptional groups of Lie type up to a few unknowns. For this
we employ the new cohomological methods of the first author, together with
properties of generalized Gelfand-Graev characters which were recently shown
to hold whenever the underlying characteristic is good.
\end{abstract}

\maketitle


\section{Introduction} \label{sec:intro}

The main aim of this paper is to determine the decomposition matrices for
the unipotent blocks of finite exceptional groups of Lie type $E_6(q)$,
$\tw2E_6(q)$, $E_7(q)$, $E_8(q)$ and $F_4(q)$ for odd primes $\ell$ dividing
$\Phi_4(q)=q^2+1$. For these groups, we study only the blocks with defect at
most $2$ --- which amounts to excluding only the principal block of $E_8(q)$
--- and we obtain an approximation to the decomposition matrices in that a
small number of entries remain undetermined for $\tw2E_6(q)$ and $F_4(q)$.
On the way, we also find decomposition matrices for orthogonal groups of rank
up to~7 which occur as Levi subgroups. As a byproduct, we obtain the
repartition of the simple unipotent modules into Harish-Chandra series.

\smallskip

The decomposition matrices are determined inductively, by the combination of
standard methods like Harish-Chandra induction and restriction, and the new
ingredient introduced in \cite{Du13} to tackle the discrete series. Our
strategy can be summarized in the following three steps.

\renewcommand{\descriptionlabel}[1]{\hspace{\labelsep}{\emph{#1}}}
\begin{description}[leftmargin=7mm,itemsep=3pt]
 \item[Step 1.] We start by using Harish-Chandra induction from proper Levi
  subgroups, Harish-Chandra restriction from suitable overgroups and
  decomposition numbers for Hecke algebras to compute the columns of the
  decomposition matrix corresponding to the non-cuspidal simple modules. In
  several cases this determines entirely the unipotent part of the
  decomposition matrix.
 \item[Step 2.] We consider suitable generalized Gelfand--Graev characters
  containing the missing columns. The properties of these projective
  characters (see \cite{Lu92,Tay14}) force the decomposition matrix to be
  unitriangular, but their construction introduce some conditions
  on the underlying characteristic of the groups considered.
 \item[Step 3.] Finally, we use virtual projective characters afforded by
  cohomology of Deligne--Lusztig varieties. As observed in \cite{Du13},
  we have some control on the multiplicity of the various PIMs in these virtual
  characters, from which we deduce upper bounds on the missing decomposition
  numbers. In many cases these bounds are small enough to determine the
  numbers.
\end{description}

The paper is built up as follows. In Section~\ref{sec:methods} we present the
general methods used in many of the arguments. Then in Section~\ref{sec:orth}
we determine decomposition matrices for some orthogonal groups of
type $D_n$, $n\le8$. In the following sections, we consider the exceptional
groups of type $E_n$, $6\le n\le 8$. For these we are able to determine all
the decomposition numbers for blocks of defect at most $2$, which excludes
only the principal block of groups of types $D_7$, $D_8$ and $E_8$. 
In Section~\ref{sec:twist} we turn to the twisted groups of types $\tw2D_n$,
$n\le7$, and $\tw2E_6$. We finish by the case of symplectic groups of type
$C_n$, $n \le4$ and exceptional groups of type $F_4$ in Section~\ref{sec:F4}.
In those cases we have to assume that the underlying characteristic is good,
and even then we are not able to determine all decomposition numbers.
Some entries remain unknown, but we still give conditions and relations that
they satisfy.
\smallskip

Let us note that our calculations give rather large examples for
Geck's conjectures on the shape of $\ell$-decomposition matrices 
(see \cite[Conjecture 3.4]{GH97} for a precise formulation of the conjectures).
Indeed,  in all the cases we consider we observe that:
\begin{itemize}
 \item the decomposition matrix has a unitriangular shape, with respect to an
  order compatible with Lusztig's $a$-function,
 \item within a given family, the square submatrix is the identity matrix
 (up to some indeterminates for types $D_7$, $\tw2D_7$, $\tw2E_6$ and $C_4$),
 \item any cuspidal unipotent character remains irreducible after
  $\ell$-reduction (more generally any unipotent character with smallest
  $a$-function within its Harish-Chandra series).
\end{itemize}

\section{Methods}   \label{sec:methods}

We determine decomposition matrices for unipotent blocks of various
families of groups of Lie type $G$, where $G=\bG(q)$ is the group of fixed
points under a Frobenius endomorphism with respect to an $\FF_q$-structure
of a simple algebraic group $\bG$ over the algebraic closure of $\FF_q$. More
precisely, we consider the case that $\ell$ is an odd prime dividing $q^2+1$.
In particular, we have $\ell\ge5$ always. In the proofs we make use of several
standard arguments which we collect here for easier reference.
\par
Firstly, the subdivision of unipotent characters into $\ell$-blocks is
known in our situation, see \cite{BMM}. Secondly, by results of Geck and
Hiss, whenever $\ell$ is a good prime for the group in question, then the
unipotent characters form a basic set for the union of unipotent blocks.
Thus, for good primes the decomposition matrix for a unipotent block is known
once the
decomposition numbers for the unipotent characters in that block have been
found. To determine the decomposition matrix of the block is hence equivalent
to finding the (unipotent parts of the) ordinary characters of all
projective indecomposable modules (PIMs) in that block.   \par
One standard method for constructing projective characters is via
Harish-Chandra induction $R_L^G$ of projective characters from proper Levi
subgroups $L$, which we may assume to be know by induction. Thus our first
source for projective characters is
\renewcommand{\descriptionlabel}[1]{\hspace{\labelsep}{\bf#1}}
\begin{description}[leftmargin=7mm]
\item[(HCi)] Harish-Chandra induction of projective characters from proper
  Levi subgroups.
\end{description}
This Harish-Chandra induction can be computed explicitly in terms of induction
in relative Weyl groups. All of our calculations were done in the
\Chevie-system \cite{Chv}. In addition, Harish-Chandra restriction $^*R_L^G$
of projective characters also yields projective characters. This leads to the
following indecomposability criterion:
\begin{description}[leftmargin=7mm]
\item[(HCr)] Let $\chi$ be a projective character of $G$. If no proper
 subcharacter of $\chi$ has the property that its Harish-Chandra restriction
 to Levi subgroups $L$ decomposes non-negatively on the PIMs of $L$, then
 $\chi$ is the character of a PIM.
\end{description}
One of our results is the subdivision into modular Harish-Chandra series of
the Brauer characters in the block. A valuable criterion to determine this
is given by \cite[Thm.~4.2]{GHM}:
\begin{description}[leftmargin=7mm]
\item[(Csp)] The group $G$ has a cuspidal unipotent Brauer character if and
  only if a Sylow $\ell$-subgroup of $G$ is not contained in any proper
  Levi subgroup of $G$.
\end{description}
Furthermore, the ordinary Gelfand--Graev character always provides the
Steinberg PIM:
\begin{description}[leftmargin=7mm]
\item[(St)] There exists a PIM with unipotent part just the ordinary Steinberg
  character. It is non-cuspidal if and only if a Sylow $\ell$-subgroup of $G$
  is contained in a proper Levi subgroup $L$. In this case it is a summand of
  the Harish-Chandra induction of the Steinberg PIM from $L$.
\end{description}
Other PIMs will appear as direct summands of generalized Gelfand--Graev
representations (GGGRs). The results of Lusztig \cite[\S11]{Lu92}, which have
recently been extended to good characteristic by Taylor (see
\cite[Thm.~14.10]{Tay14}), give an approximation of some columns of the
decomposition matrix:
\begin{description}[leftmargin=7mm]
\item[(GGGR)] Assume the underlying characteristic of $G$ is good. Given a
  unipotent character $\rho$, there exists a (projective) Gelfand-Graev
  representation $\Gamma$ such that $\rho$ occurs in $\Gamma$, and any other
  unipotent constituent in $\Gamma$ is either in the same family as $\rho$ or
  has a larger $a$-value.
\end{description}
Under suitable conditions on the family of unipotent characters containing
$\rho$ one can even use \cite[Thm. 6.5(ii)]{DLM14} to compute the
multiplicities of the characters in the family in a GGGR. An instructive
example is given in the proof of Theorem~\ref{thm:D7,d=4}.

A further tool is given by a particular case of Dipper's result (see
\cite[4.10]{Di90} for the precise assumptions):
\begin{description}[leftmargin=7mm]
\item[(End)] The decomposition matrix of the Hecke algebra
$\End_G(R_T^G(\ZZ_\ell))$ embeds as a submatrix into the decomposition
matrix of $G$.
\end{description}
Dipper's result holds more generally for a Hecke algebra associated with
$R_L^G(\rho)$ where $\rho$ is a cuspidal unipotent character satisfying the
following two conditions (see \cite[\S2.6]{Ge90}):
\begin{itemize}
 \item[-] the $\ell$-reduction of $\rho$ is an irreducible Brauer character
  $\vhi$, 
 \item[-] $N_W(W_L,\rho) = N_W(W_L,\vhi)$
\end{itemize}
where $W$ denotes the Weyl group of $G$. Note that when $L$ is classical,
$\rho$ is the unique cuspidal unipotent character, so that
$N_W(W_L,\rho) = N_W(W_L)$ and the second condition is automatically satisfied.
The first condition is conjectured to hold whenever $\ell$ is good, as already
mentioned in the introduction.

For dealing with PIMs which are not induced from a proper Levi subgroup, we
will make use of suitably chosen Deligne--Lusztig characters as in
\cite{Du13} and \cite[Sec.~6]{DM13}:
\begin{description}[leftmargin=7mm]
\item[(DL)] Let $\chi$ be a projective character of $G$. If $w\in W$ is minimal
  in the Bruhat order for the property
  that the unipotent part of $\chi$ occurs in the Deligne--Lusztig character
  $R_w$, then the sign of its multiplicity in $R_w$ is $(-1)^{\ell(w)}$.
\end{description}
We shall often use the following particular case:
\begin{description}[leftmargin=7mm]
\item[(Cox)] Let $\chi$ be the character of the projective cover
  of a cuspidal unipotent module, and $w \in W$ be a Coxeter element.
  Then the multiplicity of the unipotent part of $\chi$ in
  $(-1)^{\ell(w)} R_w$ is non-negative.
\end{description}
The previous two arguments will usually give upper bounds on
decomposition numbers. Lower bounds can be obtained from $\ell$-reduction
of non-unipotent characters, which are non-negative combinations of irreducible
Brauer characters. This applies in particular to the Deligne--Lusztig induction
of characters in general position:
\begin{description}[leftmargin=7mm]
\item[(Red)] Let $T_w$ be a torus of type $w \in W$. Assume that there exists
  an $\ell$-character of $T_w$ in general position. Then the $\ell$-reduction
  of $(-1)^{\ell(w)} R_w$ is a non-negative combination of irreducible
  Brauer characters.
\end{description}
Indeed, if $\theta$ is an $\ell$-character of $T_w$ in general position then
$(-1)^{\ell(w)} R_{T_w}^G(\theta)$ is an irreducible character by \cite{DL76}
and it has the same $\ell$-reduction as $(-1)^{\ell(w)} R_w = (-1)^{\ell(w)}
R_{T_w}^G(1_{T_w})$.

We will also use variations of the following elementary observation:
\begin{description}[leftmargin=7mm]
\item[(Sum)] Let $\chi_1+\chi_2,\chi_1+\chi_3,\chi_2,\chi_3$ be characters of
  projectives modules, and assume that $\chi_2,\chi_3$ are indecomposable. Then
  $\chi_1$ is the character of a projective module.
\end{description}
Indeed, we have two direct sum decompositions of a projective module with
character $\chi_1+\chi_2+\chi_3$, and the theorem of Krull--Schmidt allows to
conclude. Sometimes, we will also make use of the following obvious fact:
\begin{description}[leftmargin=7mm]
\item[(Deg)] The degres of the irreducible Brauer characters in a block can be
  computed from the inverse of the decomposition matrix; they are all positive.
\end{description}

Our notation for modular Harish-Chandra series is as follows: characters in the
principal series are labelled "ps", or sometimes "p" for short in large tables.
If a Levi subgroup has a single cuspidal Brauer character, its Harish-Chandra
series is labelled by the Dynkin type of that Levi subgroup. Else, it is
labelled by the name of the corresponding ordinary unipotent character. For
these, in turn, as customary we use the labelling in terms of ordinary
Harish-Chandra series. Cuspidal Brauer characters are labelled by "c".

\section{Decomposition matrices for orthogonal groups of type $D_n$}   \label{sec:orth}

We first determine the decomposition matrices of orthogonal groups
$\SO_{2n}^+(q)$, with $n\le7$ and $q$ a prime power, for primes
$2\ne\ell|(q^2+1)$, except in the case that $(q^2+1)_\ell=5$. In the
latter case, the decomposition matrices can be expected to be different from
those in the general case, see Remarks~\ref{rem:D4} and~\ref{rem:D5}.\par

\subsection{Decomposition matrices for $\SO_8^+(q)$}
Let first $n=4$, so $G=\SO_8^+(q)$. For of the unipotent characters of $G$
lie in $\ell$-blocks of defect zero, the others lie in the principal block.
Miyachi \cite[Lemma~9]{Mi08} gives an approximation of the $\Phi_4$-modular
decomposition matrix of the principal block of $\SO_8^+(q)$ when $q$ is odd,
based on the triangularity of the decomposition matrix proved by Geck--Pfeiffer
\cite{GP92} using generalized Gelfand--Graev characters. Here we extend their
results to all $q$ and determine the missing entry.

\begin{thm}   \label{thm:D4,d=4}
 Let $\ell$ be a prime. The $\ell$-modular decomposition matrix for the
 principal block of $\SO_8^+(q)$, $\ell|(q^2+1)$ with $(q^2+1)_\ell>5$, is
 as given in Table~\ref{tab:D4,d=4}.
\end{thm}

\begin{table}[htbp]
\caption{$\SO_8^+(q)$, $(q^2+1)_\ell>5$}   \label{tab:D4,d=4}
$$\vbox{\offinterlineskip\halign{$#$\hfil\ \vrule height10pt depth4pt&&
      \hfil\ $#$\hfil\cr
   .4&                1& 1\cr
  .31&      q^2\Ph3\Ph6& 1& 1\cr
   2+&      q^2\Ph3\Ph6& 1& .& 1\cr
   2-&      q^2\Ph3\Ph6& 1& .& .& 1\cr
1.21&\hlf q^3\Ph2^4\Ph6& 1& 1& 1& 1& 1\cr
 D_4&\hlf q^3\Ph1^4\Ph3& .& .& .& .& .& 1\cr
.21^2&      q^6\Ph3\Ph6& .& 1& .& .& 1& .& 1\cr
 1^2+&      q^6\Ph3\Ph6& .& .& 1& .& 1& .& .& 1\cr
 1^2-&      q^6\Ph3\Ph6& .& .& .& 1& 1& .& .& .& 1\cr
 .1^4&           q^{12}& .& .& .& .& 1& 2& 1& 1& 1& 1\cr
\noalign{\hrule}
  & & ps& ps& ps& ps& ps& c& D_3& A_3& A_3'& c\cr
  }}$$
Here, $D_3,A_3,A_3'$ denote three non-conjugate Levi subgroups of type $A_3$.
\end{table}

\begin{proof}
We obtain all projectives in the table except for those in the 6th and 10th
column by Harish-Chandra induction (HCi). The last column correspond to the
Steinberg-PIM, which is cuspidal (St). The unipotent parts of the five
principal series characters are precisely those of the Hecke algebra of type
$D_4$ at a fourth root of unity by (End). Hence they are indecomposable for
all $\ell$ dividing $q^2+1$. The printed PIMs in the series $A_3, A_3'$ and
$D_3$ are indecomposable by (HCr).
\par
Using the table of unipotent characters of $G$ in \Chevie~\cite{Chv} one finds
that the tensor product of the (projective) unipotent character $\rho_{.2^2}$
with the cuspidal unipotent character $\rho_{D_4}$ decomposes on the principal
block as $(q-1)/2 \rho_{D_4} + (q^2-1)/4 \rho_{.1^4}$ for odd $q$, respectively
as $q/2 \rho_{D_4} + q^2/4 \rho_{.1^4}$ when $q$ is even. This shows the existence
of a PIM involving only the cuspidal unipotent character $\rho_{D_4}$ and an
unknown multiple $a$ of the Steinberg character $\rho_{.1^4}$. In particular
the decomposition matrix has unitriangular shape. Let $s_1,\ldots,s_4$ denote
the simple reflections in the Weyl group of $G$. When $(q^2+1)_\ell>5$, there
exists an $\ell$-character in general position in the Sylow $\Phi_4$-torus
$T_w$ for $w =(s_1 s_2 s_3 s_4)^2$, forcing the relation $a \geq 2$ by (Red).
\par
Finally, we use (Cox) to determine $a$: the generalized $1$-eigenspace of the
Frobenius endomorphism on the Deligne--Lusztig character associated with a
Coxeter element decomposes as
$$\begin{aligned}
  P_{s_1 s_2 s_3 s_4}=&\,\rho_{.4} + \rho_{D_4} + \rho_{.1^4} \\
     =&\, \Psi_1-\Psi_2-\Psi_3-\Psi_4+2 \Psi_5 + \Psi_6 - \Psi_7 -\Psi_8
       - \Psi_9 + (2-a)\Psi_{10},\\
\end{aligned}$$
where $\Psi_i$, $i\ne6$, denotes the unipotent part of the PIM corresponding
to the $i$th column in Table~\ref{tab:D4,d=4} and
$\Psi_{6}:= \rho_{D_4} + a \rho_{.1^4}$, so that $a$ must be equal to~2.
\end{proof}

\begin{rem}   \label{rem:D4}
The $5$-modular decomposition matrices of $\SO_8^+(2)$ and $\SO_8^+(3)$ are
known; they differ from the one in Table~\ref{tab:D4,d=4} in that the entry
``$2$" is replaced by ``$1$". Thus, Theorem~\ref{thm:D4,d=4} does not extend
to the case $(q^2+1)_\ell=5$.
\end{rem}

\subsection{Decomposition matrices for $\SO_{10}^+(q)$}
We next consider the 10-dimensional orthogonal groups $G=\SO_{10}^+(q)$.
Here, $G$ has four unipotent $\ell$-blocks when $\ell|(q^2+1)$, the principal
block, one block with cyclic defect and two of defect zero, see
e.g.~\cite{BMM}.

\begin{thm}   \label{thm:D5,d=4}
 Let $\ell$ be a prime. The $\ell$-modular decomposition matrices for the
 unipotent blocks of $\SO_{10}^+(q)$ of positive defect for $(q^2+1)_\ell>5$
 are as given in Tables~\ref{tab:D5,d=4} and~\ref{tab:D5,d=4,def1}.
\end{thm}

\begin{table}[htbp]
\caption{$\SO_{10}^+(q)$, $(q^2+1)_\ell>5$}   \label{tab:D5,d=4}
$$\vbox{\offinterlineskip\halign{$#$\hfil\ \vrule height10pt depth4pt&&
      \hfil\ $#$\hfil\cr
    .5&                     1& 1\cr
   1.4&             q\Ph5\Ph6& .& 1\cr
   2.3&           q^2\Ph5\Ph8& 1& 1& 1\cr
   .32&  \hlf q^3\Ph5\Ph6\Ph8& 1& .& .& 1\cr
  1.31&  \hlf q^3\Ph3\Ph5\Ph8& .& 1& 1& .& 1\cr
 D_4:2&\hlf q^3\Ph1^4\Ph3\Ph5& .& .& .& .& .& 1\cr
 1.2^2&       q^5\Ph5\Ph6\Ph8& 1& .& 1& 1& .& .& 1\cr
 .31^2&       q^6\Ph3\Ph6\Ph8& .& .& .& .& 1& .& .& 1\cr
 .2^21&  \hlf q^7\Ph5\Ph6\Ph8& .& .& .& 1& .& .& 1& .& 1\cr
1.21^2&  \hlf q^7\Ph3\Ph5\Ph8& .& .& 1& .& 1& .& .& 1& .& 1\cr
D_4:1^2&\hlf q^7\Ph1^4\Ph3\Ph5&.& .& .& .& .& .& .& .& .& .& 1\cr
1^2.1^3&       q^{10}\Ph5\Ph8& .& .& 1& .& .& .& 1& .& .& 1& .& 1\cr
 1.1^4&        q^{13}\Ph5\Ph6& .& .& .& .& .& 2& .& 1& .& 1& .& 1& 1\cr
  .1^5&                q^{20}& .& .& .& .& .& .& 1& .& 1& .& 2& 1& .& 1\cr
\noalign{\hrule}
  & & ps& ps& ps& ps& ps& D_4& ps& D_3& D_3& ps& D_4& A_3& .1^4& .1^4\cr
  }}$$
\end{table}

\begin{table}[ht]
\caption{$\SO_{10}^+(q)$, block of defect 1, $2\ne\ell|(q^2+1)$} \label{tab:D5,d=4,def1}
$$\vbox{\offinterlineskip\halign{$#$
        \vrule height10pt depth 2pt width 0pt&& \hfil$#$\hfil\cr
 .41& \vr& 2.21& \vr& 1^2.21& \vr& .21^3& \vr& \bigcirc&\cr
 & ps& & ps& & ps& & D_3\cr
  }}$$
\end{table}

\begin{proof}
The Brauer tree for the block with cyclic defect is easily determined (see
also \cite{FS2}), so it remains to consider the principal block.
Let us denote by $\Psi_1,\ldots,\Psi_{14}$ the linear combinations of
unipotent characters given by the columns in Table~\ref{tab:D5,d=4}. We shall
show that these are the unipotent parts of projective indecomposable
characters of $G$. \par
Using (HCi) gives $\Psi_i$ for $i\in\{2,3,4,5,8,10,12\}$.
The decomposition matrix of the Hecke algebra of type $D_5$ at a fourth root
of unity gives by (End) the seven principal series PIMs $\Psi_1,\ldots,\Psi_5$,
$\Psi_7$ and $\Psi_{10}$ for all primes $\ell$ with $(q^2+1)_\ell>5$.
Furthermore, (HCi) yields $\tPsi_9:=\Psi_9+\Psi_8$ and
$\tPsi_{10}:=\Psi_9+\Psi_{10}$. An application of (Sum) yields $\Psi_9$.
\par
The centralizer of a Sylow $\Phi_4$-torus of $G$ is contained in a Levi
subgroup $L$ of type $D_4$, so \cite[Thm.~4.2]{GHM} shows that the
Harish-Chandra induction $\tPsi_{13}=\Psi_{13}+\Psi_{14}$ of the Steinberg PIM
from $L$ has two summands, namely $\Psi_{13}$ and $\Psi_{14}$. \par
(HCi) also yields a projective character with unipotent part
$\tPsi_6=\Psi_6+\Psi_{11}$. The Hecke algebra for the ordinary cuspidal
character of a Levi subgroup $L\le G$ of type $D_4$ has type $A_1$ with
parameter $q^4$, hence is semisimple modulo~$\ell$, so the Harish-Chandra
induction of the corresponding PIM from $L$ has two summands in that
Harish-Chandra
series. Decomposition of the Harish-Chandra restriction to the proper Levi
subgroups shows that these summands must have the form
$\Psi_6-(2-a)(\Psi_{13}-\Psi_{14})$ and $\Psi_{11}+(2-a)(\Psi_{13}-\Psi_{14})$,
with one undetermined parameter $a\in\{0,1,2\}$.
\par
Finally, using (HCr) we check that, independently from the value of $a$, all
$\Psi_i$ are indecomposable. Indeed, no proper subsums restrict to a
non-negative combination of PIMs in all Levi subgroups. We shall prove that
$a=2$ in Theorem~\ref{thm:D7,d=4} below.
\end{proof}

\begin{rem}   \label{rem:D5}
The $5$-modular decomposition matrix of $\SO_{10}^+(2)$ is known; it differs
from the one in Table~\ref{tab:D5,d=4} in that the two entries ``$2$" are
replaced by ``$1$"s. Thus, Theorem~\ref{thm:D5,d=4} does not extend to
the case where $(q^2+1)_\ell=5$.
\end{rem}

\subsection{Decomposition matrices for $\SO_{12}^+(q)$}
Now let $G=\SO_{12}^+(q)$. This group has four unipotent $\Phi_4$-blocks of
defect zero and three blocks of defect $\Phi_4^2$. We label these blocks by
the symbol for their $4$-Harish-Chandra source in a $4$-split Levi
subgroup of type $\SO_4^+(q)$.

\begin{thm}   \label{thm:D6,d=4}
 Let $\ell$ be a prime. The $\ell$-modular decomposition matrices for the
 unipotent blocks of $\SO_{12}^+(q)$ of positive defect, for $(q^2+1)_\ell>5$,
 are as given in Tables~\ref{tab:D6,d=4,bl1}--\ref{tab:D6,d=4,bl3}.
\end{thm}

\begin{table}[htbp]
\caption{$\SO_{12}^+(q)$, block $\binom{2}{0}$, and $\SO_{16}^+(q)$, block~$\binom{3}{1}$,
   $(q^2+1)_\ell>5$}   \label{tab:D6,d=4,bl1}
{\setlength\arraycolsep{1.5pt}
$$\begin{array}{l|ccccccccccccccc|l}
      .6&                                1& 1&  &  &  &  &  &  &  &  &  &  &  &  &  &     1.7\cr
     2.4&           q^2\Ph3\Ph5\Ph6\Ph{10}& 1& 1&  &  &  &  &  &  &  &  &  &  &  &  &     3.5\cr
    1.41&  \hlf q^3\Ph2^4\Ph3\Ph6^2\Ph{10}& .& 1& 1&  &  &  &  &  &  &  &  &  &  &  &    1.52\cr
  D_4:2.&     \hlf q^3\Ph1^4\Ph3^2\Ph5\Ph6& .& .& .& 1&  &  &  &  &  &  &  &  &  &  & D_4:3.1\cr
   .41^2&               q^6\Ph5\Ph8\Ph{10}& .& .& 1& .& 1&  &  &  &  &  &  &  &  &  &    .521\cr
    .3^2&    \hlf q^4\Ph5\Ph6^2\Ph8\Ph{10}& 1& .& .& .& .& 1&  &  &  &  &  &  &  &  &  1.3^21\cr
    2.31&    \hlf q^4\Ph3^2\Ph5\Ph8\Ph{10}& 1& 1& 1& .& .& .& 1&  &  &  &  &  &  &  &    3.32\cr
 1^2.2^2&       q^8\Ph3\Ph5\Ph6\Ph8\Ph{10}& 1& .& .& .& .& 1& 1& 1&  &  &  &  &  &  & 1^2.321\cr
  2.21^2&       q^8\Ph3^2\Ph5\Ph6^2\Ph{10}& .& .& 1& .& 1& .& 1& .& 1&  &  &  &  &  &  2^21.3\cr
 1^3.21&\hlf q^{10}\Ph2^4\Ph5\Ph6^2\Ph{10}& .& .& .& .& .& .& 1& 1& 1& 1&  &  &  &  &1^3.31^2\cr
 D_4:.2&\hlf q^{10}\Ph1^4\Ph3^2\Ph5\Ph{10}& .& .& .& .& .& .& .& .& .& .& 1&  &  &  & D_4:1.3\cr
   1^4.2&    \hlf q^{13}\Ph3\Ph5\Ph6^2\Ph8& .& .& .& 2& 1& .& .& .& 1& 1& .& 1&  &  &   1^5.3\cr
 .2^21^2& \hlf q^{13}\Ph3^2\Ph6\Ph8\Ph{10}& .& .& .& .& .& 1& .& 1& .& .& .& .& 1&  & 1.321^2\cr
   .21^4&                q^{20}\Ph5\Ph{10}& .& .& .& .& .& .& .& 1& .& 1& 2& .& 1& 1&  1.31^4\cr
\noalign{\hrule}
 \omit& & ps& ps& ps& D_4& D_3& ps& ps& ps& ps& A_3& D_4& .1^4& D_3& .1^4\cr
\end{array}$$}
\end{table}

\begin{table}[htbp]
\caption{$\SO_{12}^+(q)$, block $\binom{1}{1}$, $(q^2+1)_\ell>5$}   \label{tab:D6,d=4,bl2}
$$\vbox{\offinterlineskip\halign{$#$\hfil\ \vrule height10pt depth4pt&&
      \hfil\ $#$\hfil\cr
     1.5&                   q\Ph3\Ph6\Ph8& 1\cr
      3+&              q^3\Ph5\Ph8\Ph{10}& 1& 1\cr
      3-&              q^3\Ph5\Ph8\Ph{10}& 1& .& 1\cr
    1.32&      q^5\Ph3\Ph5\Ph6\Ph8\Ph{10}& 1& 1& 1& 1\cr
    .321& \hlf q^7\Ph2^4\Ph6^2\Ph8\Ph{10}& .& .& .& 1& 1\cr
 D_4:1.1&    \hlf q^7\Ph1^4\Ph3^2\Ph5\Ph8& .& .& .& .& .& 1\cr
  1.2^21&      q^9\Ph3\Ph5\Ph6\Ph8\Ph{10}& .& 1& 1& 1& 1& .& 1\cr
    1^3+&           q^{15}\Ph5\Ph8\Ph{10}& .& 1& .& .& .& .& 1& 1\cr
    1^3-&           q^{15}\Ph5\Ph8\Ph{10}& .& .& 1& .& .& .& 1& .& 1\cr
   1.1^5&              q^{21}\Ph3\Ph6\Ph8& .& .& .& .& 1& 2& 1& 1& 1& 1\cr
\noalign{\hrule}
 \omit& & ps& ps& ps& ps& D_3& D_4& ps& A_3& A_3& .1^4\cr
   }}$$
\end{table}

\begin{table}[htbp]
\caption{$\SO_{12}^+(q)$, block $\binom{1\,2}{0\,1}$, and $\SO_{16}^+(q)$,
    block~$\binom{1\,2\,3}{0\,1\,3}$, $(q^2+1)_\ell>5$}   \label{tab:D6,d=4,bl3}
{\setlength\arraycolsep{1.5pt}
$$\begin{array}{l|ccccccccccccccc|l}
     .51&                 q^2\Ph5\Ph{10}& 1&  &  &  &  &  &  &  &  &  &  &  &  &  &    1.51^2\cr
     .42&  \hlf q^3\Ph3^2\Ph6\Ph8\Ph{10}& 1& 1&  &  &  &  &  &  &  &  &  &  &  &  &     1.421\cr
   1^2.4&     \hlf q^3\Ph3\Ph5\Ph6^2\Ph8& .& .& 1&  &  &  &  &  &  &  &  &  &  &  &     1^3.5\cr
    21.3&\hlf q^4\Ph2^4\Ph5\Ph6^2\Ph{10}& 1& .& 1& 1&  &  &  &  &  &  &  &  &  &  &    3.31^2\cr
D_4:1^2.&\hlf q^4\Ph1^4\Ph3^2\Ph5\Ph{10}& .& .& .& .& 1&  &  &  &  &  &  &  &  &  & D_4:1^3.1\cr
   2.2^2&     q^6\Ph3\Ph5\Ph6\Ph8\Ph{10}& 1& 1& .& 1& .& 1&  &  &  &  &  &  &  &  &     2.321\cr
  1^2.31&     q^6\Ph3^2\Ph5\Ph6^2\Ph{10}& .& .& 1& 1& .& .& 1&  &  &  &  &  &  &  &    1^3.32\cr
   .31^3&          q^{12}\Ph5\Ph8\Ph{10}& .& .& .& .& .& .& 1& 1&  &  &  &  &  &  &    .321^3\cr
   .2^3&\hlf q^{10}\Ph5\Ph6^2\Ph8\Ph{10}& .& 1& .& .& .& 1& .& .& 1&  &  &  &  &  &    1.32^2\cr
1^2.21^2&\hlf q^{10}\Ph3^2\Ph5\Ph8\Ph{10}& .& .& .& 1& .& 1& 1& .& .& 1&  &  &  &  & 1^3.2^21\cr
1.21^3&\hlf q^{13}\Ph2^4\Ph3\Ph6^2\Ph{10}& .& .& .& .& 2& .& 1& 1& .& 1& 1&  &  &  & 1.2^21^3\cr
D_4:.1^2&\hlf q^{13}\Ph1^4\Ph3^2\Ph5\Ph6& .& .& .& .& .& .& .& .& .& .& .& 1&  &  & D_4:1.1^3\cr
 1^2.1^4&      q^{16}\Ph3\Ph5\Ph6\Ph{10}& .& .& .& .& 2& 1& .& .& .& 1& 1& .& 1&  &   1^3.1^5\cr
    .1^6&                         q^{30}& .& .& .& .& .&1& .& .& 1& .& .& 2& 1& 1&     1.1^7\cr
\noalign{\hrule}
 \omit& & ps& ps& ps& ps& D_4& ps& ps& D_3& D_3& ps& .1^4& D_4& A_3& .1^4\cr
\end{array}$$}
\end{table}

\begin{proof}
All projectives listed in the tables are obtained by (HCi) from the Levi
subgroups of types $D_5$ and $A_5$, except that instead of $\Psi_4$ and
$\Psi_{11}$ in the first block we obtain $\Psi_4-(2-a)(\Psi_{12}-\Psi_{14})$
and $\Psi_{11}+(2-a)(\Psi_{12}-\Psi_{14})$, and instead of $\Psi_5$ and
$\Psi_{12}$ in the third block we obtain $\Psi_5-(2-a)(\Psi_{11}-\Psi_{14})$
and $\Psi_{12}+(2-a)(\Psi_{11}-\Psi_{14})$, with the parameter $a\in\{0,1,2\}$
as in the proof of Theorem~\ref{thm:D5,d=4}. It is straightforward to check
by (HCr) that all these projective characters are indecomposable. We shall
prove that $a=2$ in Theorem~\ref{thm:D7,d=4} below.
\end{proof}

\subsection{Decomposition matrices for $\SO_{14}^+(q)$}
We now consider the three blocks of positive $\Phi_4$-defect for the
groups $G=\SO_{14}^+(q)$. The non-principal block is again labelled by the
symbol for its $4$-Harish-Chandra source in a $4$-split Levi subgroup of
type $\SO_6^+(q)$.

\begin{thm}   \label{thm:D7,d=4}
 Let $G=\SO_{14}^+(q)$ and $\ell$ a prime with $(q^2+1)_\ell>5$.
 \begin{enumerate}
  \item[\rm(a)] If $q$ is odd then the decomposition matrix for the
   principal $\ell$-block of $G$ is as given in Tables~\ref{tab:D7,d=4} and
   \ref{tab:D7,d=4b}.
  \item[\rm(b)] The decomposition matrix for the non-principal unipotent
   $\ell$-block $G$ of positive defect is as given in Table~\ref{tab:D7,d=4,bl2}.
 \end{enumerate}
\end{thm}

\begin{table}[htbp]
\caption{$\SO_{14}^+(q)$, $q$ odd, principal block, $(q^2+1)_\ell>5$}  \label{tab:D7,d=4}
$$\vbox{\offinterlineskip\halign{$#$\hfil\ \vrule height9pt depth4pt&&
      \hfil\ $#$\hfil\cr
       .7& 1\cr
      1.6& 1& 1\cr
      2.5& 1& .& 1\cr
     1.51& .& .& 1& 1\cr
    1^2.5& .& 1& .& .& 1\cr
   D_4:3.& .& .& .& .& .& 1\cr
      3.4& 1& 1& 1& .& .& .& 1\cr
     3.31& 1& 1& 1& 1& 1& .& 1& 1\cr
     1.42& .& .& 1& 1& .& .& 1& .& 1\cr
    1.3^2& 1& 1& .& .& .& .& 1& .& .& 1\cr
     2.32& 1& .& 1& 1& .& .& 1& 1& 1& .& 1\cr
 D_4:1^3.& .& .& .& .& .& .& .& .& .& .& .& 1\cr
    .51^2& .& .& .& 1& .& .& .& .& .& .& .& .& 1\cr
     .421& .& .& .& 1& .& .& .& .& 1& .& .& .& 1& 1\cr
    1^3.4& .& .& .& .& 1& .& .& .& .& .& .& .& .& .& 1\cr
  D_4:2.1& .& .& .& .& .& 1& .& .& .& .& .& .& .& .& .& 1\cr
    .3^21& 1& .& .& .& .& .& .& .& .& 1& .& .& .& .& .& .& 1\cr
   21^2.3& .& .& .& 1& 1& .& .& 1& .& .& .& .& 1& .& 1& .& .& 1\cr
D_4:1^2.1& .& .& .& .& .& .& .& .& .& .& .& 1& .& .& .& .& .& .& 1\cr
   1^2.32& 1& 1& .& .& 1& .& 1& 1& .& 1& .& .& .& .& .& .& .& .& .& 1\cr
   1^3.31& .& .& .& .& 1& .& .& 1& .& .& .& .& .& .& 1& .& .& 1& .& 1\cr
    .32^2& .& .& .& .& .& .& .& .& 1& .& 1& .& .& 1& .& .& .& .& .& .\cr
  D_4:1.2& .& .& .& .& .& .& .& .& .& .& .& .& .& .& .& .& .& .& 1& .\cr
   2.2^21& .& .& .& 1& .& .& 1& 1& 1& .& 1& .& 1& 1& .& .& .& 1& .& .\cr
   D_4:.3& .& .& .& .& .& .& .& .& .& .& .& .& .& .& .& .& .& .& .& .\cr
    1.2^3& .& .& .& .& .& .& 1& .& 1& .& 1& .& .& 1& .& .& .& .& .& .\cr
 1^2.2^21& 1& .& .& .& .& .& 1& 1& .& 1& 1& .& .& .& .& .& 1& .& .& 1\cr
    1^4.3& .& .& .& .& .& 2& .& .& .& .& .& .& 1& .& 1& .& .& 1& .& .\cr
   .321^2& .& .& .& .& .& .& .& .& .& 1& .& .& .& .& .& .& 1& .& .& 1\cr
D_4:1.1^2& .& .& .& .& .& .& .& .& .& .& .& .& .& .& .& 1& .& .& .& .\cr
 1.2^21^2& .& .& .& .& .& .& 1& .& .& 1& .& 2& .& .& .& .& 1& .& .& 1\cr
 1^3.21^2& .& .& .& .& .& .& 1& 1& .& .& 1& .& .& .& .& .& .& 1& .& 1\cr
    .31^4& .& .& .& .& .& .& .& .& .& .& .& .& .& .& .& .& .& .& .& 1\cr
   1.21^4& .& .& .& .& .& .& .& .& .& .& .& 2& .& .& .& .& 1& .& 2& 1\cr
    1^5.2& .& .& .& .& .& 2& .& .& .& .& .& .& 1& 1& .& 2& .& 1& .& .\cr
 D_4:.1^3& .& .& .& .& .& .& .& .& .& .& .& .& .& .& .& .& .& .& .& .\cr
  1^3.1^4& .& .& .& .& .& .& 1& .& .& .& 1& 2& .& .& .& .& .& .& .& .\cr
  1^2.1^5& .& .& .& .& .& .& .& .& .& .& 1& 2& .& .& .& .& 1& .& 2& .\cr
    1.1^6& .& .& .& .& .& .& .& .& .& .& 1& .& .& 1& .& 2& .& .& .& .\cr
     .1^7& .& .& .& .& .& .& .& .& .& .& 1& .& .& .& .& .& .& .& .& .\cr
\noalign{\hrule}
  & ps& ps& ps& ps& ps& D_4& ps& ps& ps& ps& ps& D_4& D_3& D_3& ps& D_4& D_3& ps& D_4& ps& \cr
   }}$$
\end{table}

\begin{table}[ht]
\caption{$\SO_{14}^+(q)$, principal block, cntd.}   \label{tab:D7,d=4b}
$$\vbox{\offinterlineskip\halign{$#$\hfil\ \vrule height10pt depth4pt&&
      \hfil\ $#$\hfil\cr
   1^3.31& 1\cr
    .32^2& .& 1\cr
  D_4:1.2& .& .& 1\cr
   2.2^21& .& .& .& 1\cr
   D_4:.3& .& .& 1& .&  1&\cr
    1.2^3& .& 1& .& 1&  b_1& 1 &\cr
 1^2.2^21& .& .& .& .&    .& c& 1\cr
    1^4.3& 1& .& .& .&  b_3& .& .& 1\cr
   .321^2& .& .& .& .&  b_4& .& .& .& 1\cr
D_4:1.1^2& .& .& .& .&  b_5& .& .& .& .& 1\cr
 1.2^21^2& .& .& .& .&  b_6& c& 1& .& 1& .& 1\cr
 1^3.21^2& 1& .& .& 1&  b_7& c& 1& .& .& .& .& 1\cr
    .31^4& 1& .& 2& .&  b_8& .& .& .& 1& .& .& .& 1\cr
   1.21^4& 1& .& 2& .&  b_9& c& 1& .& 1& .& 1& 1& 1& 1\cr
    1^5.2& 1& .& .& 1& b_{10}& .& .& 1& .& .& .& 1& .& .& 1\cr
 D_4:.1^3& .& .& .& .& b_{11}& .& .& .& .& 1& .& .& .& .& .& 1\cr
  1^3.1^4& .& .& .& 1& b_{12}& c\pl1& 1& .& .& .& 1& 1& .& .& .& .& 1\cr
  1^2.1^5& .& .& .& .& b_{13}& c& 1& .& .& .& 1& 1& .& 1& .& .& 1& 1\cr
    1.1^6& .& 1& .& 1& b_{14}& 1& .& .& .& 2& .& 1& .& .& 1& .& 1& d& 1\cr
     .1^7& .& 1& .& .& b_{15}& 1& .& .& .& 2& .& .& .& .& .& 2& 1& d\pl3& 1& 1\cr
\noalign{\hrule}
  & A_3& D_3& D_4& ps& c& A_3D_3\!& ps& .1^4& D_3& D_4& .1^4& A_3& .1^4& .1^4& .1^4& c& A_3& c& .1^4& c\cr
   }}$$
 Here, $b_1,c,d \in \{0,1\}$. Moreover, if $p>5$ then $b_3=b_4=b_5=b_7= 0$
 and $b_6\in\{0,1\}$.
\end{table}

\begin{table}[ht]
\caption{$\SO_{14}^+(q)$, block $\binom{1\,3}{0\,1}$, $(q^2+1)_\ell>5$}   \label{tab:D7,d=4,bl2}
$$\vbox{\offinterlineskip\halign{$#$\hfil\ \vrule height10pt depth4pt&&
      \hfil\ $#$\hfil\cr
    .61&                  \Ph3\Ph6& 1\cr
   21.4&     \hlf q^2\Ph5\Ph6\Ph7\Ph8& 1& 1\cr
    .43&  \hlf q^2\Ph6\Ph7\Ph8\Ph{10}& 1& .& 1\cr
D_4:21.&   \hlf q^2\Ph1^4\Ph3\Ph5\Ph7& .& .& .& 1\cr
 1^2.41&        q^4\Ph3\Ph6^2\Ph7\Ph8& .& 1& .& .& 1\cr
  21.31&   q^5\Ph3\Ph5\Ph6\Ph7\Ph{10}& 1& 1& .& .& 1& 1\cr
  21.22&       q^7\Ph5\Ph7\Ph8\Ph{10}& 1& .& 1& .& .& 1& 1\cr
21.21^2&   q^9\Ph3\Ph5\Ph6\Ph7\Ph{10}& .& .& .& .& 1& 1& 1& 1\cr
  .41^3&        q^{10}\Ph5\Ph8\Ph{10}& .& .& .& .& 1& .& .& .& 1\cr
 2.21^3&     q^{12}\Ph3\Ph6^2\Ph7\Ph8& .& .& .& 2& 1& .& .& 1& 1& 1\cr
 1^4.21&  \hlf q^{14}\Ph5\Ph6\Ph7\Ph8& .& .& .& 2& .& .& 1& 1& .& 1& 1\cr
 .2221&\hlf q^{14}\Ph6\Ph7\Ph8\Ph{10}& .& .& 1& .& .& .& 1& .& .& .& .& 1\cr
D_4:.21&\hlf q^{14}\Ph1^4\Ph3\Ph5\Ph7& .& .& .& .& .& .& .& .& .& .& .& .& 1\cr
  .21^5&               q^{28}\Ph3\Ph6& .& .& .& .&.& .& 1& .& .& .& 1& 1& 2& 1\cr
\noalign{\hrule}
 \omit& & ps& ps& ps& D_4& ps& ps& ps& ps& D_3& .1^4& A_3& D_3& D_4& .1^4\cr
   }}$$
Here, all character degrees have been divided by $q^2\Ph4\Ph{12}$.
\end{table}

\begin{proof}
Let us first consider the block with defect $\Phi_4^2$. We argue how to
construct projectives $\Psi_1,\ldots,\Psi_{14}$ with unipotent part equal to
the columns in Table~\ref{tab:D7,d=4,bl2}. (HCi) and (End)
give all $\Psi_i$, except that instead of $\Psi_4$ and $\Psi_{13}$ we find
$\Psi_4-(2-a)(\Psi_{10}-\Psi_{14})$ and $\Psi_{13}+(2-a)(\Psi_{10}-\Psi_{14})$,
with $a\in\{1,2\}$ as in the proof of Theorem~\ref{thm:D6,d=4}. Again,
it is easily seen by (HCr) that all these characters are indecomposable. Now
the tenth Brauer character has positive degree only if $a\ge2$, which shows by
(Deg) that $a=2$ in this table and also in the decomposition matrices for
$\SO_{10}^+(q)$ and $\SO_{12}^+(q)$, thus completing the proofs of
Theorems~\ref{thm:D5,d=4} and~\ref{thm:D6,d=4}. Note that up to this point we
did not use any assumption on $q$.
\par
We now turn to the principal block. Here, (HCi) yields
$\Psi_i$ except for $i\in\{2,14,16,19,$ $24,25,26,34,35,36,38,40\}$.
Using (Sum) we get other projective characters:
$\Psi_2+\Psi_3$ and $\Psi_2+\Psi_4+\Psi_{10}$ give $\Psi_2$,
$\Psi_{14}+\Psi_{17}$ and $\Psi_{14}+\Psi_{29}$ give $\Psi_{14}$,
$\Psi_6+\Psi_{19}$ and $\Psi_{19}+\Psi_{30}$ give $\Psi_{19}$,
$\Psi_{12}+\Psi_{16}$ and $\Psi_{16}+\Psi_{19}$ give $\Psi_{16}$,
$\Psi_9+\Psi_{20}+\Psi_{24}$ and $\Psi_{24}+\Psi_{27}$ give $\Psi_{24}$,
$\Psi_{31}+\Psi_{35}$ and $\Psi_{33}+\Psi_{35}$ give $\Psi_{35}$,
$\Psi_{28}+\Psi_{34}$ and $\Psi_{34}+\Psi_{35}$ give $\Psi_{34}$.
Furthermore we find $\Psi_{26}$ with $c=1$. The PIM $\Psi_{40}$ is cuspidal by
(St). (HCr) shows that all of the projectives obtained so far, with the
possible exception of $\Psi_{26}$, which might contain $\Psi_{27}$ once, are
indecomposable. We have thus obtained all but three columns of the
decomposition matrix. Since we have accounted for all proper Harish-Chandra
series, the remaining three Brauer characters must be cuspidal.
\par
To establish the unitriangularity we  look at suitably chosen generalized
Gelfand--Graev representations and use (GGGR). Note that we have to assume
that $q$ is odd in order to construct these representations and use the results
in \cite{Tay14}. Let us first consider the family
$\cF = \{\rho_{1^5.2}, \rho_{1.21^4},\rho_{.2^21^3},\rho_{D_4:.1^3}\}$
of unipotent characters. The special character of this family is
$\rho_{1.21^4}$; via the Springer correspondence, it corresponds to a special
unipotent class, and we denote by $\cO$ its dual. By \cite[Thm.~14.10]{Tay14},
the character of any GGGR attached to $\cO$ involves characters lying in $\cF$
or in a family with a strictly larger $a$-value than that of $\cF$.
In particular, the only characters in the block that can occur are
$\rho_{1^5.2}, \rho_{1.21^4},\rho_{D_4:.1^3},\rho_{1^3.1^4},\rho_{1^2.1^5},
\rho_{1.1^6}$ and $\rho_{.1^7}$, which gives an approximation of $\Psi_{36}$.
Furthermore, $u\in\cO^F$ satisfies the following two conditions:
\begin{itemize}
 \item the small finite group attached to the family as in \cite[Chap.~4]{Lu84}
  and the component group $A_G(u):=(C_\bG(u)/C_\bG(u)^\circ)^F$ are isomorphic
  (to $\ZZ/2\ZZ$),
 \item at most one of the local systems on $(u)$ is not in the principal block,
\end{itemize}
in which case one can apply \cite[Thm.~6.5(ii)]{DLM14} to compute the projection
of the GGGR $\Gamma_u$ into the span of $\cF$. Recall that the characters in
a family are parametrized by pairs $(g,\psi)$ where $g$ runs over a set of
representatives of conjugacy classes of the small finite group attached to
the family, say $\cG$ and $\psi \in \Irr(C_\cG(g))$. For $g \in \cG$,
the Mellin transform of the pair $(g,1)$ is given by
$$ \mu_{(g,1)} = \sum_{\psi \in \Irr(C_\cG(g))} \psi(g)
   \rho_{(g,\psi)}.$$
In particular, the small finite group for the dual family of $\cF$ is
$\ZZ/2\ZZ$ and the Mellin transforms of $(1,1)$ and $(-1,1)$ are
$$\begin{aligned}
\mu_{(1,1)} = \rho_{(1,1)} + \rho_{(1,\vareps)} = \rho_{1.51} + \rho_{1^2.5}, \\
\mu_{(-1,1)} = \rho_{(-1,1)} - \rho_{(-1,\vareps)} = \rho_{.52} - \rho_{D_4:3.},\\
\end{aligned}$$
where $\vareps$ denotes the non-trivial character of $\ZZ/2\ZZ$.
By \cite[Thm.~6.5(ii)]{DLM14}, the projections to $\cF$ of the two GGGRs attached
to $\cO$ are given by the Alvis--Curtis duals of these characters, that is by
$$\rho_{1.21^4} + \rho_{1^5.2} \quad \text{and} \quad
  \rho_{.2^21^3} + \rho_{D_4:.1^3}.$$
Taking the second GGGR and cutting by the block, we obtain a projective
character whose unipotent part is given by
$\rho_{D_4:.1^3}+c_1\rho_{1^3.1^4}+c_2\rho_{1^2.1^5}+c_3\rho_{1.1^6}
+c_4\rho_{.1^7}$.
Note that $c_1$ is actually zero  by \cite[Thm.~14.10]{Tay14} since $\rho_{1^3.1^4}$
and $\rho_{D_4:.1^3}$ have the same $a$-value but lie in different families.
\par
For $\Psi_{25}$ and $\Psi_{38}$ we consider the GGGRs associated with the
families $\{\rho_{1^2.2^21},\rho_{1^3.2^2},\rho_{1.2^3},$ $\rho_{D_4:.3}\}$
and $\{\rho_{1^2.1^5}\}$, from which we deduce that the decomposition matrix
is unitriangular. Moreover, if we denote by $(b_i)_{i=1,\ldots,15}$
(resp.~$c_5,c_6$) the unknown entries in the 25th (resp.~38th) column then
\cite[Thm.~6.5]{DLM14} yields $b_2 = 0$.
\par
The unipotent part of the Gelfand-Graev representation of $\SO_{16}^+(q)$
associated with the family $\{\rho_{1^2.1^6}\}$, cut by the principal block,
is of the form $\rho_{1^2.1^6} + \alpha \rho_{2.1^6} + \beta \rho_{.1^8}$ for
suitable $\alpha,\beta\ge0$. The Harish-Chandra restriction of this character
to $\SO_{14}^+(q)$, cut by the principal block, equals
$\rho_{1^2.1^5} + \rho_{1.1^6} + (\alpha + \beta) \rho_{.1^7}$
and thus forces $c_5 \leq 1$.
\par
If moreover $p>5$ we may also consider a GGGR associated to the family
$\{\phi_{700,42},$ $ \phi_{400,43},\phi_{300,44},D_4:\phi_{1,12}''\}$
of $E_8$, whose projection on the family is $\phi_{400,43}+D_4:\phi_{1,12}''$.
By \cite[Thm.~14.10]{Tay14}, the only other unipotent characters lying in the principal
$\ell$-block which can occur as constituents are $\phi_{1,120}, \phi_{35,74},
\phi_{50,56},\phi_{210,52},\phi_{567,46},\phi_{112,63}$ and $D_4:\phi_{1,24}$.
As above, the Harish-Chandra reduction of the GGGR gives upper bounds for
some of the $b_i$'s, namely $b_3=b_4=b_5 = b_7= 0$ and $b_6 \leq 1$
\par
We now use (DL) to obtain relations on the other decomposition numbers.
Let $w$ be a Coxeter element. For $v < w$, one checks easily that
the characters $\Psi_{38}$, $\Psi_{39}$ and $\Psi_{40}$ do not occur
in $R_v$. Therefore the computation of $R_w$ yields by (DL) three inequalities
which are $ -c_2 \geq 0$,  $c_2c_5 -c_3 \geq 0$ and
$2+c_3-c_4+c_2(c_6-c_5) \geq 0$. This forces $c_2 = c_3 = 0$ and
$c_4 \leq 2$. We use (Red) to prove that $c_4 = 2$. More precisely, we consider
the $\ell$-reduction of a non-unipotent character which is obtained by inducing
an $\ell$-character in general position of a $\Phi_4$-torus (of order
$(q+1)(q^2+1)^3$). Such a character exists whenever $(q^2+1)_\ell>12$, which
automatically holds if $(q^2+1)_\ell>5$.
This yields the relations $c_4 \geq 2$ and $c_6 \geq c_5+3$, so that $c_4 = 2$.
In particular, none of $\Psi_{38}$, $\Psi_{39}$ and $\Psi_{40}$ occur in $R_w$.
\par
Finally, we use (DL) with  $w'= s_1s_3s_1s_2s_3s_4s_5s_6s_7$, where
$s_1,\ldots,s_7$ are the simple reflections ordered as in \Chevie\ (so the
end nodes are $1,2$ and~$7$), and we find
$c_5+3 \geq c_6$, so that $c_6 = c_5+3$. Note that one relation on the $b_i$'s
can also be obtained. The resulting decomposition matrix is given in
Table~\ref{tab:D7,d=4b}, where $c_5\in\{0,1\}$ is simply denoted by $d$.
\end{proof}

\begin{rem}   \label{rem:kawanaka}
Under some assumptions on the special unipotent class $\cO$, Kawanaka
conjectured in \cite{Ka87} that one can decompose any GGGR associated with $\cO$
into a sum of projective characters, each of which contains only one unipotent
character of the family. If Kawanaka's conjecture holds for the family
$\{1^2.2^21,1^3.2^2,1.2^3,D_4:.3\}$ of $\SO_{14}(q)$, then $b_1$ and $d$ must
be equal to zero. More generally, as suggested by Geck, Kawanaka's characters
should force block unitriangularity of the decomposition matrix whenever $p$ is
good.
\end{rem}

Our method is not sufficient to determine all the decomposition numbers
of $\SO_{14}^+(q)$. However, we can use \cite[Conj.~1.2]{DM14} to determine
small upper bounds for the missing entries. Following \cite{DM14}, we denote
by $Q_w$ the virtual character afforded by the Alvis--Curtis dual of the
intersection cohomology of the Deligne--Lusztig variety corresponding to $w$.
In addition, if $\lambda \in \overline{\FF}_\ell$ we consider the virtual
character $Q_w[\lambda]$ afforded by the generalized $\lambda$-eigenspace of
the Frobenius on the intersection cohomology. Up to a sign, $Q_w[\lambda]$ is
a proper character and Conjecture~1.2 in \cite{DM14} predicts that it is
actually the unipotent part of a projective character. The multiplicities of
the various PIMs in $Q_w[\lambda]$ depend on the decomposition numbers
(including the missing entries), forcing some linear combinations of
decomposition numbers to be non-negative.

\begin{prop} \label{prop:conjForD7}
 Assume Conjecture~1.2 in \cite{DM14} holds. Then in the decomposition matrix
 of the principal $\Phi_4$-block of $\SO_{14}^+(q)$, we have
 $b_1 =b_3=b_4=b_5=b_8 =0$, $b_6,b_7,b_{10},b_{11} \leq 2$, $b_9 \leq 6$,
 $b_{12},b_{13} \leq 12$, $b_{14} \leq 18$ and $b_{15}\leq 20$.
\end{prop}

\begin{proof}
To obtain the upper bounds on the $b_i$'s we consider the characters
$Q_{w_1}[q^3]$ for
$w_1 = s_1s_2s_3s_1s_2s_3s_5s_4s_3s_6s_5 s_4s_3s_7s_6s_5s_4s_3$ and
$Q_{w_2}[1]$ for $w_2 = s_1s_2s_3s_1s_2s_3s_5s_6s_5s_7s_6s_5s_4$ as well as
their decomposition on the basis of PIMs.
\par
The coefficient of $\Psi_{26}$ in $Q_{w_1}[q^3]$ is $10-14b_1$. By
\cite[Conj.~1.2]{DM14} it must be non-negative, which forces $b_1 = 0$.
The list of coefficients of $\Psi_{27},\ldots,\Psi_{32}$ on $Q_{w_2}[1]$ is
given by $-7b_3,\ -7b_4,\ -7b_5,\, 17+7b_4-7b_6,\, 17-7b_7$ and $7b_4-7b_8$.
Since
they must be all non-negative, we get $b_3=b_4=b_5=b_8=0$ and $b_6,b_7 \leq 2$.
With these values, the coefficients of $\Psi_{33},\ldots,\Psi_{40}$ on $Q_{w_2}[1]$ are
$$\begin{aligned}
  & 14+7b_6+7b_7-7b_9,\, 7b_7-7b_{10},\, 19-7b_{11},\\
  & 62+7b_6+7b_7-7b_{12},\, -9-7b_6-7b_7+7b_9+7b_{12}-7b_{13},\\
  & 60-7b_6-7b_7+7b_{10}+7b_{12}-7b_{14}+d(9+7b_6+7b_7-7b_9-7b_{12}+7b_{13}),\\
  & 42+21b_6+21b_7-21b_9-7b_{10}+14b_{11}-21b_{12}+21b_{13}+7b_{14}-7b_{15}.\\
\end{aligned}$$
>From the first line we deduce $b_9 \leq b_6 +b_7 +2$, $b_{10} \leq b_7$ and
$b_{11} \leq 2$. From the second, $b_{12} \leq b_6+b_7+8$ and
$b_{13} \leq -2-b_6-b_7+b_9+b_{12}$, so that $b_{13} \leq b_{12}$.
Finally, the last two lines yield $b_{14} \leq 8-b_6-b_7+b_{10}+b_{12}
\leq 16 + b_{10}$ and $b_{15} \leq 6+3b_6+3b_7-3b_9-b_{10}+2b_{11}-3b_{12}+3b_{13}+b_{14}
\leq 16+2b_{11}$.
\end{proof}

\subsection{Decomposition matrices for $\SO_{16}^+(q)$}
Finally, we consider the three blocks of $\Phi_4$-defect $\Phi_4^2$ for the
groups $G=\SO_{16}^+(q)$. They are again labelled by the symbol for their
$4$-Harish-Chandra source in a $4$-split Levi subgroup of type $\SO_8^+(q)$.

\begin{thm}   \label{thm:D8,d=4}
 The $\ell$-modular decomposition matrices for the unipotent blocks of
 $\SO_{16}^+(q)$ of defect $\Phi_4^2$, for primes $\ell$ with $(q^2+1)_\ell>5$,
 are as given in Tables~\ref{tab:D6,d=4,bl1}, \ref{tab:D6,d=4,bl3},
 \ref{tab:D8,d=4,bl2} and~\ref{tab:D8,d=4,bl3}.  \par
 The two blocks of $\SO_{12}^+(q)$ and $\SO_{16}^+(q)$ in
 Table~\ref{tab:D6,d=4,bl1} are Morita equivalent, as well as the two blocks
 in Table~\ref{tab:D6,d=4,bl3}.
\end{thm}

\begin{table}[ht]
\caption{$\SO_{16}^+(q)$, block~$\binom{1\,2}{0\,3}$, $(q^2+1)_\ell>5$}   \label{tab:D8,d=4,bl2}
$$\vbox{\offinterlineskip\halign{$#$\hfil\ \vrule height10pt depth4pt&&
      \hfil\ $#$\hfil\cr
   1^2.6&                    \hlf \Ph8\Ph{10}& 1\cr
    2.51&               \hlf q\Ph5\Ph8\Ph{14}& .& 1\cr
    3.41&    \hlf q^2\Ph2^4\Ph6\Ph{10}\Ph{14}& 1& 1& 1\cr
    2.42&    \hlf q^3\Ph3^2\Ph8\Ph{10}\Ph{14}& .& 1& 1& 1\cr
 1^2.3^2&  \hlf q^6\Ph5\Ph6\Ph8\Ph{10}\Ph{14}& 1& .& 1& .& 1\cr
   .42^2&        \hlf q^7\Ph8^2\Ph{10}\Ph{14}& .& .& .& 1& .& 1\cr
D_4:1^2.2&\hlf q^8\Ph1^4\Ph3^2\Ph5\Ph6\Ph{14}& .& .& .& .& .& .& 1\cr
 2.2^3& \hlf q^{10}\Ph5\Ph6\Ph8\Ph{10}\Ph{14}& .& .& 1& 1& .& 1& .& 1\cr
D_4:2.1^2& \hlf q^{10}\Ph1^4\Ph3^2\Ph5\Ph{10}& .& .& .& .& .& .& .& .& 1\cr
 .3^21^2&     \hlf q^{11}\Ph8^2\Ph{10}\Ph{14}& .& .& .& .& 1& .& .& .& .& 1\cr
1^2.2^21^2&\hlf q^{15}\Ph3^2\Ph8\Ph{10}\Ph{14}&.& .& 1& .& 1& .& .& .& .& 1& 1\cr
1^3.21^3& \hlf q^{18}\Ph2^4\Ph6\Ph{10}\Ph{14}& .& .& 1& .& .& .& .& 1& .& .& 1& 1\cr
1^2.21^4&          \hlf q^{21}\Ph5\Ph8\Ph{14}& .& .& .& .& .& .& 2& .& .& 1& 1& 1& 1\cr
   1^6.2&              \hlf q^{28}\Ph8\Ph{10}& .& .& .& .& .& 1& .& 1& 2& .& .& 1& .& 1\cr
\noalign{\hrule}
 \omit& & ps& ps& ps& ps& ps& D_3& D_4& ps& D_4& D_3& ps& A_3& .1^4& .1^4 \cr
   }}$$
Here, all character degrees have been divided by $q^3\Ph4^2\Ph7\Ph{12}$.
\end{table}

\begin{table}[ht]
\caption{$\SO_{16}^+(q)$, block~$\binom{2\,3}{0\,1}$, $(q^2+1)_\ell>5$}   \label{tab:D8,d=4,bl3}
$$\vbox{\offinterlineskip\halign{$#$\hfil\ \vrule height10pt depth4pt&&
      \hfil\ $#$\hfil\cr
     .62&                     \hlf \Ph5\Ph8& 1\cr
     .53&             \hlf q\Ph7\Ph8\Ph{10}& 1& 1\cr
D_4:2^2.&        \hlf q^2\Ph1^4\Ph3\Ph5\Ph7& .& .& 1\cr
   2^2.4&        \hlf q^3\Ph5\Ph6^2\Ph7\Ph8& 1& .& .& 1\cr
  2^2.31&   \hlf q^6\Ph3\Ph5\Ph7\Ph8\Ph{10}& 1& 1& .& 1& 1\cr
1^2.41^2&            \hlf q^7\Ph5\Ph7\Ph8^2& .& .& .& 1& .& 1\cr
21.31^2&\hlf q^8\Ph2^4\Ph3\Ph6^2\Ph7\Ph{10}& .& .& .& 1& 1& 1& 1\cr
  1.41^3&\hlf q^{10}\Ph2^4\Ph5\Ph6^2\Ph{10}& .& .& .& .& .& 1& .& 1\cr
21^2.2^2&\hlf q^{10}\Ph3\Ph5\Ph7\Ph8\Ph{10}& .& 1& .& .& 1& .& 1& .& 1\cr
  2.31^3&         \hlf q^{11}\Ph5\Ph7\Ph8^2& .& .& 2& .& .& 1& 1& 1& .& 1\cr
 1^4.2^2&     \hlf q^{15}\Ph5\Ph6^2\Ph7\Ph8& .& .& 2& .& .& .& 1& .& 1& 1& 1\cr
D_4:.2^2&     \hlf q^{18}\Ph1^4\Ph3\Ph5\Ph7& .& .& .& .& .& .& .& .& .& .& .& 1\cr
 .2^31^2&        \hlf q^{21}\Ph7\Ph8\Ph{10}& .& 1& .& .& .& .& .& .& 1& .& .& .& 1\cr
 .2^21^4&               \hlf q^{28}\Ph5\Ph8& .& .& .& .& .& .& .& .& 1& .& 1& 2& 1& 1\cr
\noalign{\hrule}
 \omit& & ps& ps& D_4& ps& ps& ps& ps& D_3& ps& .1^4& A_3& D_4& D_3& .1^4 \cr
   }}$$
Here, all character degrees have been divided by $q^3\Ph4^2\Ph{12}\Ph{14}$.
\end{table}

\begin{proof}
Harish-Chandra induction sends the fourteen PIMs in the principal block
of $\SO_{12}^+(q)$ to the fourteen listed projectives of the first block of
$\SO_{16}^+(q)$ in Table~\ref{tab:D6,d=4,bl1}, and it sends irreducible
characters to irreducible characters. Thus, by \cite[Thm.~0.2]{Br90} those
two blocks are Morita equivalent. Exactly the same assertions hold for
the third block of $\SO_{12}^+(q)$ and the fourth block of $\SO_{16}^+(q)$
in Table~\ref{tab:D6,d=4,bl3}.
\par
In the second block of $G$, (HCi) yields all columns in
Table~\ref{tab:D8,d=4,bl2}, except for the second one. This is then obtained
from the projectives $\Psi_1+\Psi_2$ and $\Psi_2+\Psi_4$ via (Sum).  \par
In the third block, we obtain all $\Psi_i$ in Table~\ref{tab:D8,d=4,bl3} for
$i\notin\{1,5,6,9\}$.
Then, using (Sum) $\Psi_4+\Psi_6$ and $\Psi_6+\Psi_7$ give $\Psi_6$,
$\Psi_2+\Psi_9$ and $\Psi_7+\Psi_9+\Psi_{11}$ give $\Psi_9$,
$\Psi_4+\Psi_5$ and $\Psi_5+\Psi_9$ give $\Psi_5$, and finally
$\Psi_1+\Psi_2$ and $\Psi_1+\Psi_5$ give $\Psi_1$.
\end{proof}

\section{Decomposition matrices for $E_6(q)$}

We now turn to decomposition matrices for the exceptional Lie type groups.
We first consider $G=E_6(q)$ for primes $\ell|\Phi_4(q)$ in which case the
Sylow $\ell$-subgroups are
abelian homocyclic of rank~2. Note that again we do not need and will not
specify the isogeny type of $G$, since the decomposition numbers of the
unipotent characters will not depend on such a choice.
The group $E_6(q)$ has ten unipotent blocks of $\ell$-defect zero, one of
cyclic defect and the principal block containing 16 unipotent characters.
\par
Here, the decomposition matrix of the principal block has been
determined by Miyachi \cite[Thm.~37]{Mi08} except for three missing entries,
which coincide with entries of the decomposition matrix for $D_4(q)$, again
under the assumption that $q$ is a power of a good prime. We give an
independent proof of his result, valid for all prime powers $q$, and find
the remaining entries using Theorem~\ref{thm:D4,d=4}:

\begin{prop}   \label{prop:E6,d=4}
 Let $\ell$ be a prime. Then the $\ell$-modular decomposition matrices for
 the unipotent blocks of $E_6(q)$ of positive defect, for $(q^2+1)_\ell>5$,
 are as given in Tables~\ref{tab:E6,d=4} and~\ref{tab:E6,d=4,def1}.
 In particular, the three undetermined entries in the $\ell$-modular
 decomposition matrix of $E_6(q)$ in \cite[Thm.~37]{Mi08} are all equal to~2.
\end{prop}

\begin{table}[ht]
\caption{$E_6(q)$, $(q^2+1)_\ell>5$}   \label{tab:E6,d=4}
$$\vbox{\offinterlineskip\halign{$#$\hfil\ \vrule height10pt depth4pt&&
      \hfil\ $#$\hfil\cr
  \phi_{1,0}&                                  1& 1\cr
  \phi_{6,1}&                          q\Ph8\Ph9& .& 1\cr
 \phi_{15,5}&         \hlf q^3\Ph5\Ph6^2\Ph8\Ph9& .& .& 1\cr
 \phi_{15,4}&        \hlf q^3\Ph5\Ph8\Ph9\Ph{12}& 1& 1& .& 1\cr
       D_4:3&       \hlf q^3\Ph1^4\Ph3^2\Ph5\Ph9& .& .& .& .& 1\cr
 \phi_{81,6}&         q^6\Ph3^3\Ph6^2\Ph9\Ph{12}& .& 1& 1& .& .& 1\cr
 \phi_{80,7}&  \sxt q^7\Ph2^4\Ph5\Ph8\Ph9\Ph{12}& .& 1& .& 1& .& 1& 1\cr
 \phi_{10,9}& \thrd q^7\Ph5\Ph6^2\Ph8\Ph9\Ph{12}& 1& .& .& 1& .& .& .& 1\cr
\phi_{90,8}&\thrd q^7\Ph3^3\Ph5\Ph6^2\Ph8\Ph{12}& .& .& 1& .& .& 1& .& .& 1\cr
      D_4:21&   \hlf q^7\Ph1^4\Ph3^2\Ph5\Ph8\Ph9& .& .& .& .& .& .& .& .& .& 1\cr
\phi_{81,10}&      q^{10}\Ph3^3\Ph6^2\Ph9\Ph{12}& .& .& .& .& .& 1& 1& .& 1& .& 1\cr
\phi_{15,17}&      \hlf q^{15}\Ph5\Ph6^2\Ph8\Ph9& .& .& .& .& 2& .& .& .& 1& .&1& 1\cr
\phi_{15,16}&     \hlf q^{15}\Ph5\Ph8\Ph9\Ph{12}& .& .& .& 1& .& .& 1& 1& .& .& .& .& 1\cr
     D_4:1^3&    \hlf q^{15}\Ph1^4\Ph3^2\Ph5\Ph9& .& .& .& .& .& .& .& .& .& .& .& .& .& 1\cr
 \phi_{6,25}&                     q^{25}\Ph8\Ph9& .& .& .& .& .& .& 1& .& .& 2& 1& .& 1& .& 1\cr
 \phi_{1,36}&                             q^{36}& .& .& .& .& .& .& .& 1& .& .& .& .& 1& 2& .& 1\cr
\noalign{\hrule}
 \omit& & ps& ps& ps& ps& D_4& ps& ps& ps& ps& D_4& A_3& .1^4& A_3& D_4& .1^4& .1^4\cr
   }}$$
\end{table}

\begin{table}[ht]
\caption{$E_6(q)$, block of defect 1, $2\ne\ell|(q^2+1)$}   \label{tab:E6,d=4,def1}
$$\vbox{\offinterlineskip\halign{$#$
        \vrule height10pt depth 2pt width 0pt&& \hfil$#$\hfil\cr
 \phi_{20,2}& \vr& \phi_{60,5}& \vr& \phi_{60,11}& \vr& \phi_{20,20}& \vr& \bigcirc\cr
 & ps& & ps& & ps& & A_3\cr
  }}$$
\end{table}

\begin{proof}
The Brauer tree for the block with cyclic defect is easily determined. We will
construct PIMs $\Psi_1,\ldots,\Psi_{16}$ with unipotent parts as given by the
columns of Table~\ref{tab:E6,d=4}.\par
The PIMs in the principal series can be read off from the $\Phi_4$-modular
decomposition matrix of the Iwahori--Hecke algebra of type $E_6$ given in
\cite[Tab.~7.13]{GJ11}. Note that by \cite[Thm.~3.10]{GM09} this agrees with
the $\ell$-modular decomposition matrix whenever $\ell\ge5$ (since $5$ is a
good prime for $E_6$ and $20$ does not divide any degree of $E_6$). So we
have columns $i$ in Table~\ref{tab:E6,d=4} for $i\in\{1,2,3,4,6,7,8,9\}$.
Furthermore, $\Psi_{11}$ and $\Psi_{13}$ are Harish-Chandra induced from Levi
subgroups of types $D_5$ and $A_5$. \par
The Hecke algebra for the ordinary cuspidal unipotent character of $D_4$ is
of type $A_2$ with parameter $q$, thus remains semisimple modulo~$\ell$. Now
Harish-Chandra induction from the Levi subgroup $L$ of type $D_5$ yields
projective characters $\Psi_5+\Psi_{10}$ and $\Psi_{10}+\Psi_{14}$. Both of
these must be the sum of two projective characters. Using (HCr) we find that
$\Psi_5,\Psi_{10},\Psi_{14}$ are the only subsums which can be projective.
\par
Since all other Harish-Chandra series have been accounted for, and $G$ cannot
have cuspidal Brauer characters by (Csp), the missing three PIMs
must lie in the series of the cuspidal Brauer character $\rho_{.1^4}$ of
$D_4$. Its relative Weyl group is the symmetric group $\fS_3$, so the
corresponding Hecke algebra must be semisimple. Now Harish-Chandra induction
yields $\Psi_{12}+\Psi_{15}$ and $\Psi_{15}+\Psi_{16}$, and via (HCr) there is
a unique way for each of these to split into sums of two non-zero projective
characters.
This completes the construction of $\Psi_1,\ldots,\Psi_{16}$ and thus the
proof.
\end{proof}

\section{Decomposition matrices for $E_7(q)$}

We next consider the four unipotent $\Phi_4$-blocks of $E_7(q)$ of positive
defect (see \cite[Tab.~2]{BMM}), which we name by their $4$-Harish-Chandra
sources in a Levi subgroup of type $A_1(q)^3$.

\begin{thm}   \label{thm:E7,d=4}
 The $\ell$-modular decomposition matrices for the unipotent blocks of
 $E_7(q)$ of positive defect for primes $\ell$ with $(q^2+1)_\ell>5$ are as
 given in Tables~\ref{tab:E7,d=4,bl1}--\ref{tab:E7,d=4,bl4}.
\end{thm}

\begin{table}[ht]
\caption{$E_7(q)$, block $2\otimes 2\otimes 2$, $(q^2+1)_\ell>5$}   \label{tab:E7,d=4,bl1}
$$\vbox{\offinterlineskip\halign{$#$\hfil\ \vrule height10pt depth4pt&&
      \hfil\ $#$\hfil\cr
  \phi_{1,0}&                                                         1& 1\cr
 \phi_{56,3}&               \hlf q^3\Ph2^4\Ph6^2\Ph7\Ph{10}\Ph14\Ph{18}& .& 1\cr
      D_4:3.&                   \hlf q^3\Ph1^4\Ph3^2\Ph5\Ph7\Ph9\Ph{14}& .& .& 1\cr
\phi_{210,6}&                  q^6\Ph5\Ph7\Ph8\Ph9\Ph{10}\Ph{14}\Ph{18}& .& 1& .& 1\cr
\phi_{105,6}&               q^6\Ph5\Ph7\Ph9\Ph{10}\Ph{12}\Ph{14}\Ph{18}& 1& 1& .& .& 1\cr
\phi_{405,8}&     \hlf q^8\Ph3^3\Ph5\Ph6^2\Ph8\Ph9\Ph{12}\Ph{14}\Ph{18}& .& 1& .& 1& 1& 1\cr
\phi_{189,10}&    \hlf q^8\Ph3^2\Ph6^3\Ph7\Ph8\Ph9\Ph{10}\Ph{12}\Ph{18}& .& .& .& 1& .& .& 1\cr
\phi_{336,11}& \hlf q^{10}\Ph2^4\Ph6^2\Ph7\Ph8\Ph9\Ph{10}\Ph{14}\Ph{18}& .& .& .& 1& .& 1& 1& 1\cr
 D_4\!:\!21.&     \hlf q^{10}\Ph1^4\Ph3^2\Ph5\Ph7\Ph8\Ph9\Ph{14}\Ph{18}& .& .& .& .& .& .& .& .& 1\cr
\phi_{315,16}&\sxt q^{16}\Ph3^3\Ph5\Ph7\Ph8\Ph{10}\Ph{12}\Ph{14}\Ph{18}& .& .& .& .& 1& 1& .& 1& .& 1\cr
\phi_{35,22}&    \sxt q^{16}\Ph5\Ph6^3\Ph7\Ph8\Ph9\Ph{10}\Ph{12}\Ph{14}& .& .& 2& .& .& .& 1& 1& .& .& 1\cr
\phi_{70,18}&  \thrd q^{16}\Ph5\Ph7\Ph8\Ph9\Ph{10}\Ph{12}\Ph{14}\Ph{18}& 1& .& .& .& 1& .& .& .& .& .& .& 1\cr
\phi_{189,22}&          q^{22}\Ph3^2\Ph6^2\Ph7\Ph9\Ph{12}\Ph{14}\Ph{18}& .& .& .& .& 1& .& .& .& .& 1& .& 1& 1\cr
\phi_{120,25}&     \hlf q^{25}\Ph2^4\Ph5\Ph6^2\Ph9\Ph{10}\Ph{14}\Ph{18}& .& .& .& .& .& .& .& 1& 2& 1& .& .& 1& 1\cr
D_4\!:\!1^3.&         \hlf q^{25}\Ph1^4\Ph3^2\Ph5\Ph7\Ph9\Ph{10}\Ph{18}& .& .& .& .& .& .& .& .& .& .& .& .& .& .& 1\cr
\phi_{21,36}&                              q^{36}\Ph7\Ph9\Ph{14}\Ph{18}& .& .& .& .& .& .& .& .& .& .& .& 1& 1& .& 2& 1\cr
\noalign{\hrule}%
 \omit& & ps& ps& D_4& ps& ps& ps& ps& A_3& D_4& ps& .1^4& ps& A_3& .1^4& D_4& .1^4\cr
   }}$$
\vskip 1pc
\caption{$E_7(q)$, block $2\otimes 2\otimes 1^2$, $(q^2+1)_\ell>5$}   \label{tab:E7,d=4,bl2}
$$\vbox{\offinterlineskip\halign{$#$\hfil\ \vrule height10pt depth4pt&&
      \hfil\ $#$\hfil\cr
   \phi_{7,1}&                                          q\Ph7\Ph{12}\Ph{14}& 1\cr
  \phi_{15,7}&             \hlf q^4\Ph5\Ph8\Ph9\Ph{10}\Ph{12}\Ph{14}\Ph{18}& 1& 1\cr
 \phi_{105,5}&                \hlf q^4\Ph5\Ph7\Ph8\Ph9\Ph{10}\Ph{12}\Ph{18}& .& .& 1\cr
 \phi_{189,7}&                 q^7\Ph3^2\Ph6^2\Ph7\Ph9\Ph{12}\Ph{14}\Ph{18}& .& .& 1& 1\cr
 \phi_{280,8}&     \hlf q^7\Ph2^4\Ph5\Ph6^3\Ph7\Ph{10}\Ph{12}\Ph{14}\Ph{18}& 1& .& 1& .& 1\cr
      D_4:2.1&        \hlf q^7\Ph1^4\Ph3^3\Ph5\Ph7\Ph9\Ph{10}\Ph{12}\Ph{14}& .& .& .& .& .& 1\cr
 \phi_{378,9}&             q^9\Ph3^2\Ph6^2\Ph7\Ph8\Ph9\Ph{12}\Ph{14}\Ph{18}& .& .& 1& 1& 1& .& 1\cr
\phi_{210,13}&           q^{13}\Ph5\Ph7\Ph8\Ph9\Ph{10}\Ph{12}\Ph{14}\Ph{18}& 1& 1& .& .& 1& .& .& 1\cr
\phi_{105,15}&               q^{15}\Ph5\Ph7\Ph9\Ph{10}\Ph{12}\Ph{14}\Ph{18}& .& .& .& 1& .& .& 1& .& 1\cr
\phi_{216,16}&\hlf q^{15}\Ph2^4\Ph3^2\Ph6^3\Ph9\Ph{10}\Ph{12}\Ph{14}\Ph{18}& .& .& .& .& 1& .& 1& 1& .& 1\cr
    D_4:1^2.1&      \hlf q^{15}\Ph1^4\Ph3^3\Ph5\Ph6^2\Ph7\Ph9\Ph{12}\Ph{18}& .& .& .& .& .& .& .& .& .& .& 1\cr
 \phi_{35,31}&                 \hlf q^{30}\Ph5\Ph7\Ph8\Ph{12}\Ph{14}\Ph{18}& .& 1& .& .& .& .& .& 1& .& .& .& 1\cr
 \phi_{21,33}&                 \hlf q^{30}\Ph7\Ph8\Ph9\Ph{10}\Ph{12}\Ph{14}& .& .& .& .& .& 2& 1& .& 1& 1&  .& .& 1\cr
 \phi_{27,37}&                         q^{37}\Ph3^2\Ph6^2\Ph9\Ph{12}\Ph{18}& .& .& .& .& .& .&.& 1& .& 1& 2& 1& .& 1\cr
\noalign{\hrule}%
 \omit& & ps& ps& ps& ps& ps& D_4& ps& ps& A_3& A_3& D_4& A_3& .1^4& .1^4\cr
   }}$$
\end{table}

\begin{table}[ht]
\caption{$E_7(q)$, block $2\otimes 1^2\otimes 1^2$, $(q^2+1)_\ell>5$}   \label{tab:E7,d=4,bl3}
$$\vbox{\offinterlineskip\halign{$#$\hfil\ \vrule height10pt depth4pt&&
      \hfil\ $#$\hfil\cr
  \phi_{27,2}&                          q^2\Ph3^2\Ph6^2\Ph9\Ph{12}\Ph{18}& 1\cr
  \phi_{35,4}&                  \hlf q^3\Ph5\Ph7\Ph8\Ph{12}\Ph{14}\Ph{18}& 1& 1\cr
  \phi_{21,6}&                  \hlf q^3\Ph7\Ph8\Ph9\Ph{10}\Ph{12}\Ph{14}& .& .& 1\cr
 \phi_{216,9}& \hlf q^8\Ph2^4\Ph3^2\Ph6^3\Ph9\Ph{10}\Ph{12}\Ph{14}\Ph{18}& 1& .& 1& 1\cr
      D_4:1.2&       \hlf q^8\Ph1^4\Ph3^3\Ph5\Ph6^2\Ph7\Ph9\Ph{12}\Ph{18}& .& .& .& .& 1\cr
\phi_{210,10}&         q^{10}\Ph5\Ph7\Ph8\Ph9\Ph{10}\Ph{12}\Ph{14}\Ph{18}& 1& 1& .& 1& .& 1\cr
\phi_{105,12}&             q^{12}\Ph5\Ph7\Ph9\Ph{10}\Ph{12}\Ph{14}\Ph{18}& .& .& 1& .& .& .& 1\cr
\phi_{378,14}&        q^{14}\Ph3^2\Ph6^2\Ph7\Ph8\Ph9\Ph{12}\Ph{14}\Ph{18}& .& .& 1& 1& .& .& 1& 1\cr
\phi_{280,17}&\hlf q^{16}\Ph2^4\Ph5\Ph6^3\Ph7\Ph{10}\Ph{12}\Ph{14}\Ph{18}& .& .& .& 1& .& 1& .& 1& 1\cr
    D_4:1.1^2&   \hlf q^{16}\Ph1^4\Ph3^3\Ph5\Ph7\Ph9\Ph{10}\Ph{12}\Ph{14}& .& .& .& .& .& .& .& .& .& 1\cr
\phi_{189,20}&            q^{20}\Ph3^2\Ph6^2\Ph7\Ph9\Ph{12}\Ph{14}\Ph{18}& .& .& .& .& .& .& 1& 1& .& .& 1\cr
 \phi_{15,28}&        \hlf q^{25}\Ph5\Ph8\Ph9\Ph{10}\Ph{12}\Ph{14}\Ph{18}& .& 1& .& .& .& 1& .& .& .& .& .& 1\cr
\phi_{105,26}&           \hlf q^{25}\Ph5\Ph7\Ph8\Ph9\Ph{10}\Ph{12}\Ph{18}& .& .& .& .& 2& .& .& 1& 1& .&1& .& 1\cr
  \phi_{7,46}&                                   q^{46}\Ph7\Ph{12}\Ph{14}& .& .& .& .& .&1& .& .& 1& 2& .& 1& .& 1\cr
\noalign{\hrule}%
 \omit& & ps& ps& ps& ps& D_4& ps& ps& ps& A_3& D_4& A_3& A_3& .1^4& .1^4\cr
   }}$$
\vskip 1pc
\caption{$E_7(q)$, block $1^2\otimes 1^2\otimes 1^2$, $(q^2+1)_\ell>5$}   \label{tab:E7,d=4,bl4}
$$\vbox{\offinterlineskip\halign{$#$\hfil\ \vrule height10pt depth4pt&&
      \hfil\ $#$\hfil\cr
  \phi_{21,3}&                               q^3\Ph7\Ph9\Ph{14}\Ph{18}& 1\cr
 \phi_{120,4}&       \hlf q^4\Ph2^4\Ph5\Ph6^2\Ph9\Ph{10}\Ph{14}\Ph{18}& .& 1\cr
      D_4:.3 &          \hlf q^4\Ph1^4\Ph3^2\Ph5\Ph7\Ph9\Ph{10}\Ph{18}& .& .& 1\cr
 \phi_{189,5}&            q^5\Ph3^2\Ph6^2\Ph7\Ph9\Ph{12}\Ph{14}\Ph{18}& 1& 1& .& 1\cr
 \phi_{315,7}&  \sxt q^7\Ph3^3\Ph5\Ph7\Ph8\Ph{10}\Ph{12}\Ph{14}\Ph{18}& .& 1& .& 1& 1\cr
  \phi_{70,9}&   \thrd q^7\Ph5\Ph7\Ph8\Ph9\Ph{10}\Ph{12}\Ph{14}\Ph{18}& 1& .& .& 1& .& 1\cr
 \phi_{35,13}&     \sxt q^7\Ph5\Ph6^3\Ph7\Ph8\Ph9\Ph{10}\Ph{12}\Ph{14}& .& .& .& .& .& .& 1\cr
\phi_{336,14}&\hlf q^{13}\Ph2^4\Ph6^2\Ph7\Ph8\Ph9\Ph{10}\Ph{14}\Ph{18}& .& .& .& .& 1& .& 1& 1\cr
      D_4:.21&   \hlf q^{13}\Ph1^4\Ph3^2\Ph5\Ph7\Ph8\Ph9\Ph{14}\Ph{18}& .& .& .& .& .& .& .& .& 1\cr
\phi_{405,15}&\hlf q^{15}\Ph3^3\Ph5\Ph6^2\Ph8\Ph9\Ph{12}\Ph{14}\Ph{18}& .& .& .& 1& 1& .& .& 1& .& 1\cr
\phi_{189,17}&\hlf q^{15}\Ph3^2\Ph6^3\Ph7\Ph8\Ph9\Ph{10}\Ph{12}\Ph{18}& .& .& 2& .& .& .& 1& 1& .&.& 1\cr
\phi_{105,21}&          q^{21}\Ph5\Ph7\Ph9\Ph{10}\Ph{12}\Ph{14}\Ph{18}& .& .& .& 1& .& 1& .& .& .& 1& .& 1\cr
\phi_{210,21}&             q^{21}\Ph5\Ph7\Ph8\Ph9\Ph{10}\Ph{14}\Ph{18}& .& .& 2& .& .& .& .& 1& .&1& 1& .& 1\cr
 \phi_{56,30}&        \hlf q^{30}\Ph2^4\Ph6^2\Ph7\Ph{10}\Ph{14}\Ph{18}& .& .& .& .& .& .& .& .& 2& 1& .& 1& 1& 1\cr
     D_4:.1^3&              \hlf q^{30}\Ph1^4\Ph3^2\Ph5\Ph7\Ph9\Ph{14}& .& .& .& .& .& .& .& .& .& .& .& .& .& .& 1\cr
  \phi_{1,63}&                                                  q^{63}& .& .& .& .& .& 1& .& .& .& .& .& 1& .& .& 2& 1\cr
\noalign{\hrule}%
 \omit& & ps& ps& D_4& ps& ps& ps& ps& ps& D_4& ps& .1^4& A_3& A_3& .1^4& D_4& .1^4\cr
   }}$$
\end{table}

\begin{proof}
For the principal block, all
columns but the sixth are obtained by (HCi). The
projectives $\Psi_6+\Psi_7$ and $\Psi_6+\Psi_{10}$ then yield $\Psi_6$ via
(Sum). \par
For the second block, all but the ninth column are gotten by (HCi)
and then $\Psi_9+\Psi_{10}$ and $\Psi_9+\Psi_{12}$ give $\Psi_9$.
\par
For the third block, all PIMs are Harish-Chandra induced. Finally, all
projectives in the fourth block except for the fifth come from (HCi). The
projectives $\Psi_5+\Psi_7$ and $\Psi_5+\Psi_{10}$ then give $\Psi_5$ via
(Sum). Then (HCr) shows that all of these projectives are indecomposable.
\end{proof}


\section{Decomposition matrices for $E_8(q)$}

Finally, we determine the decomposition matrices of the unipotent
$\Phi_4$-blocks of $G=E_8(q)$ of non-maximal defect. Since $5$ is a bad prime
for $G$, and hence the basic set results do not apply in this case, we assume
that $\ell\ne5$ throughout. There are four unipotent
$\ell$-blocks of $\Phi_4$-defect two, see \cite[Tab.~2]{BMM}. These are
labelled by the four $\Phi_4$-cuspidal unipotent characters of the
$\Phi_4$-split Levi subgroup of type $D_4$.

\begin{thm}   \label{thm:E8,d=4}
 The $\ell$-modular decomposition matrices for the unipotent blocks of
 $E_8(q)$ of defect $\Phi_4^2$ for primes $\ell>5$ dividing $q^2+1$ are
 as given in Tables~\ref{tab:E8,d=4,bl1}--\ref{tab:E8,d=4,bl4}.
\end{thm}

\begin{table}[ht]
\caption{$E_8(q)$, block $\binom{3}{1}$, $5<\ell|(q^2+1)$}   \label{tab:E8,d=4,bl1}
$$\vbox{\offinterlineskip\halign{$#$\hfil\ \vrule height10pt depth4pt&&
      \hfil\ $#$\hfil\cr
      \phi_{8,1}& 1\cr
    \phi_{560,5}& .& 1\cr
   \phi_{1344,8}& .& 1& 1\cr
  D_4:\phi_{4,1}& .& .& .& 1\cr
  \phi_{1400,11}& 1& 1& .& .& 1\cr
   \phi_{840,13}& .& .& 1& .& .& 1\cr
  \phi_{4536,13}& .& 1& 1& .& 1& .& 1\cr
  \phi_{3200,16}& .& .& 1& .& .& 1& 1& 1\cr
D_4:\phi_{8,3}''& .& .& .& .& .& .& .& .& 1\cr
  \phi_{4200,21}& .& .& .& .& 1& .& 1& 1& .& 1\cr
  \phi_{2240,28}& 1& .& .& .& 1& .& .& .& .& .& 1\cr
D_4:\phi_{4,7}''& .& .& .& .& .& .& .& .& .& .& .& 1\cr
  \phi_{3240,31}& .& .& .& .& 1& .& .& .& .& 1& 1& .& 1\cr
  \phi_{1400,37}& .& .& .& .& .& .& .& .& .& .& 1& 2& 1& 1\cr
  \phi_{1008,39}& .& .& .& .& .& .& .& 1& 2& 1& .& .& 1& .& 1\cr
    \phi_{56,49}& .& .& .& 2& .& 1& .& 1& .& .& .& .& .& .& .& 1\cr
\noalign{\hrule}%
 \omit& ps& ps& ps& D_4& ps& ps& ps& A_3& D_4& ps& ps& D_4& A_3& .1^4& .1^4& .1^4\cr
   }}$$
\vskip 1pc
\caption{$E_8(q)$, block $\binom{1\,2}{0\,3}$, $5<\ell|(q^2+1)$}   \label{tab:E8,d=4,bl2}
$$\vbox{\offinterlineskip\halign{$#$\hfil\ \vrule height10pt depth4pt&&
      \hfil\ $#$\hfil\cr
      \phi_{84,4}& 1\cr
  D_4:\phi_{2,4}'& .& 1\cr
     \phi_{700,6}& 1& .& 1\cr
   \phi_{2268,10}& .& .& 1& 1\cr
   \phi_{4200,12}& 1& .& 1& 1& 1\cr
   \phi_{2100,16}& .& .& .& 1& .& 1\cr
    \phi_{448,25}& .& 2& .& .& .& 1& 1\cr
   \phi_{2016,19}& 1& .& .& .& 1& .& .& 1\cr
   \phi_{5600,19}& .& .& .& 1& 1& 1& .& .& 1\cr
   D_4:\phi_{4,8}& .& .& .& .& .& .& .& .& .& 1\cr
   \phi_{4200,24}& .& .& .& .& 1& .& .& 1& 1& .& 1\cr
   \phi_{2100,28}& .& 2& .& .& .& 1& 1& .& 1& .& .& 1\cr
   \phi_{2268,30}& .& .& .& .& .& .& .& .& 1& 2& 1& 1& 1\cr
    \phi_{700,42}& .& .& .& .& .& .& .& 1& .& 2& 1& .& 1& 1\cr
D_4:\phi_{2,16}''& .& .& .& .& .& .& .& .& .& .& .& .& .& .& 1\cr
     \phi_{84,64}& .& .& .& .& .& .& .& 1& .& .& .& .& .& 1& 2& 1\cr
\noalign{\hrule}%
 \omit& ps& D_4& ps& ps& ps& ps& .1^4& ps& ps& D_4& ps& A_3& .1^4& A_3& D_4& .1^4\cr
   }}$$
\end{table}

\begin{table}[ht]
\caption{$E_8(q)$, block $\binom{2\,3}{0\,1}$, $5<\ell|(q^2+1)$}   \label{tab:E8,d=4,bl3}
$$\vbox{\offinterlineskip\halign{$#$\hfil\ \vrule height10pt depth4pt&&
      \hfil\ $#$\hfil\cr
     \phi_{28,8}& 1\cr
    \phi_{160,7}& .& 1\cr
    \phi_{300,8}& .& .& 1\cr
   \phi_{972,12}& .& 1& 1& 1\cr
   \phi_{840,14}& .& .& 1& 1& 1\cr
   \phi_{700,16}& 1& 1& .& .& .& 1\cr
 D_4:\phi_{12,4}& .& .& .& .& .& .& 1\cr
 D_4:\phi_{6,6}'& .& .& .& .& .& .& .& 1\cr
D_4:\phi_{6,6}''& .& .& .& .& .& .& .& .& 1\cr
  \phi_{1344,19}& .& 1& .& 1& .& 1& .& .& .& 1\cr
   \phi_{840,26}& 1& .& .& .& .& 1& .& .& .& .& 1\cr
   \phi_{700,28}& .& .& .& 1& 1& .& .& .& .& 1& .& 1\cr
   \phi_{972,32}& .& .& .& .& .& 1& .& .& .& 1& 1& .& 1\cr
   \phi_{300,44}& .& .& .& .& .& .& .& .& 2& .& 1& .& 1& 1\cr
   \phi_{160,55}& .& .& .& .& .& .& 2& .& .& 1& .& 1& 1& .& 1\cr
    \phi_{28,68}& .& .& .& .& 1& .& .& 2& .& .& .& 1& .& .& .& 1\cr
\noalign{\hrule}%
 \omit& ps& ps& ps& ps& ps& ps& D_4& D_4& D_4& ps& ps& A_3& A_3& .1^4& .1^4& .1^4\cr
   }}$$
\vskip 1pc
\caption{$E_8(q)$, block $\binom{1\,2\,3}{0\,1\,3}$, $5<\ell|(q^2+1)$}   \label{tab:E8,d=4,bl4}
$$\vbox{\offinterlineskip\halign{$#$\hfil\ \vrule height10pt depth4pt&&
      \hfil\ $#$\hfil\cr
   \phi_{56,19}& 1\cr
  \phi_{1400,7}& .& 1\cr
  \phi_{1008,9}& .& .& 1\cr
  \phi_{3240,9}& .& 1& 1& 1\cr
 \phi_{2240,10}& .& 1& .& 1& 1\cr
D_4:\phi_{4,7}'& .& .& .& .& .& 1\cr
 \phi_{4200,15}& .& .& 1& 1& .& .& 1\cr
 \phi_{3200,22}& 1& .& .& .& .& .& 1& 1\cr
D_4:\phi_{8,9}'& .& .& .& .& .& .& .& .& 1\cr
 \phi_{4536,23}& .& .& .& 1& .& .& 1& 1& .& 1\cr
 \phi_{1400,29}& .& .& .& 1& 1& .& .& .& .& 1& 1\cr
  \phi_{840,31}& 1& .& .& .& .& 2& .& 1& .& .& .& 1\cr
 \phi_{1344,38}& .& .& .& .& .& 2& .& 1& .& 1& .& 1& 1\cr
D_4:\phi_{4,13}& .& .& .& .& .& .& .& .& .& .& .& .& .& 1\cr
  \phi_{560,47}& .& .& .& .& .& .& .& .& 2& 1& 1& .& 1& .& 1\cr
    \phi_{8,91}& .& .& .& .& 1& .& .& .& .& .& 1& .& .& 2& .& 1\cr
\noalign{\hrule}%
 \omit& ps& ps& ps& ps& ps& D_4& ps& ps& D_4& ps& A_3& .1^4& A_3& D_4& .1^4& .1^4\cr
   }}$$
\end{table}

\begin{proof}
In the block above $\binom{3}{1}$, all columns $\Psi_i$ except for
$i\in\{7,9,10,15\}$ are obtained by (HCi). Using (Sum) the
projectives $\Psi_1+\Psi_2+\Psi_7$ and $\Psi_7+\Psi_{11}$ yield $\Psi_7$,
$\Psi_4+\Psi_9$ and $\Psi_9+\Psi_{12}$ yield $\Psi_9$,
$\Psi_6+2\Psi_7+\Psi_{10}$ and $\Psi_{10}+\Psi_{11}$ yield $\Psi_{10}$,
and $\Psi_{11}+\Psi_{15}$ and $\Psi_{14}+\Psi_{15}$ yield $\Psi_{15}$.
\par
In the block above $\binom{1\,2}{0\,3}$, (HCi) yields all columns
$\Psi_i$ except for indices $i\in\{3,6,11\}$. Here, the projectives
$\Psi_{11}+\Psi_{12}$ and $\Psi_{11}+\Psi_{14}$ yield $\Psi_{11}$,
$\Psi_3+\Psi_5+\Psi_{11}$ and $2\Psi_3$ yield $\Psi_3$, and
$\Psi_6+2\Psi_9$ and $2\Psi_6$ yield $\Psi_6$.
\par
In the block above $\binom{2\,3}{0\,1}$, (HCi) yields the columns
$\Psi_i$ with $i\in\{3,4,6,11,12,13\}$. The information obtained in this way
does not seem to yield strong enough conditions to determine the decomposition
matrix completely. So here we use in addition the decomposition matrix of the
Hecke algebra of type $E_8$. In characteristic zero this can be found in
\cite[Tab.~7.15]{GJ11}. By \cite[Thm.~3.10]{GM09} this agrees with the
decomposition matrix in characteristic~$\ell$ for all $\ell\ge7$. This yields
in addition the principal series PIMs $\Psi_i$ with $i\in\{1,2,5,10\}$. To
construct the missing projectives, let us first consider the characters in the
Harish-Chandra series of the ordinary cuspidal character of a Levi subgroup
of type $D_4$. Here, the relative Weyl group $W$ has type $F_4$, and the Hecke
algebra $H$ has parameters $q^4,q$. But then all characters of $H$ relevant for
our block lie in semisimple blocks of $H$ for all $\ell\ge5$ by
\cite[Thm.~3.13]{Bre94}. Thus, the decomposition of induced projectives
from a Levi subgroup of type $E_7$ can be read off from the character table of
$W$. In the third block, we find $\Psi_7+\Psi_8$, $\Psi_7+\Psi_9$, which
should have a projective summand in common. The only splitting of these
projectives compatible with (HCr) is as given. This accounts for the
projectives $\Psi_7,\Psi_8,\Psi_9$ in the $D_4$-series.
Finally, we consider the characters above the cuspidal character $.1^4$ of the
Levi subgroup of type $D_4$. Here again, the relative Weyl group has type
$F_4$; the parameters of the corresponding Hecke algebra are already
determined locally inside $D_5$ and $D_4A_1$ to be the same as for the
ordinary cuspidal character considered above. So again by
\cite[Thm.~3.13]{Bre94}, all characters in the $.1^4$-Harish-Chandra
series correspond to blocks of $H$ of defect~0. Harish-Chandra induction
from $E_7$ yields $\Psi_{14}+\Psi_{15}$ and $\Psi_{15}+\Psi_{16}$, with a
common summand. Visible, the only possible splitting is as claimed.
\par
Finally, in the block indexed by $\binom{1\,2\,3}{0\,1\,3}$, (HCi)
yields all columns $\Psi_i$ except for $i\in\{2,3,7\}$. The projectives
$\Psi_5+\Psi_7+\Psi_{10}$ and $2\Psi_7+\Psi_8$ yield $\Psi_7$,
$\Psi_2+\Psi_7$ and $2\Psi_2+\Psi_4$ yield $\Psi_2$, and
$\Psi_1+\Psi_3+2\Psi_7$ and $\Psi_2+\Psi_3$ give $\Psi_3$.
\par
It is now a routine computation using (HCr) to check that none of the
projectives constructed above can be decomposable.
\end{proof}


\subsection{The $\Phi_4$-blocks in untwisted groups}
In the following table, we have collected some numerical information on the
various $\Phi_4$-blocks of defect $\Phi_4^2$ and $\Phi_4^3$ whose decomposition
matrices we have determined: the relative Weyl group $W_G(b)$ of the block
(see \cite[Tab.~1]{BMM}) and the distribution of its Brauer characters into
Harish-Chandra series:

\begin{table}[ht]
\caption{HC-series in $\Phi_4$-blocks of defect $\Phi_4^2$ and $\Phi_4^3$}  \label{tab:4-blocks}
$$\begin{array}{rc|cc|ccccccc}
 G& b& W_G(b)& |\IBr(b)|& ps& A_3& D_3& D_4& .1^4& A_3D_3& c\\
\hline
    D_4&    & G(4,2,2)& 10& 5& 2& 1& 1& 1\\
    D_6&   2&         &   & 5& 2& 1& 1& 1\\
\hline
    D_5&    & G(4,1,2)& 14& 7& 1& 2& 2& 2\\
    D_6& 1,3&         &   & 7& 1& 2& 2& 2\\
    D_7&   2&         &   & 7& 1& 2& 2& 2\\
    D_8& \text{1--4}& &   & 7& 1& 2& 2& 2\\
    E_7& 2,3&         &   & 7& \span3\quad & 2& 2\\
\hline
    E_6&    &      G_8& 16& 8& \span2\quad& 3& 3\\
    E_7& 1,4&         &   & 8& \span2\quad& 3& 3\\
    E_8& \text{1--4}& &   & 8& \span2\quad& 3& 3\\
\hline
    D_7&   1& G(4,1,3)& 40& 15& 3& 5& 6& 6& 1& 4\\
\end{array}$$
Note that the series $A_3$ and $D_3$ fuse in $E_6$ (and hence in $E_7$ and
$E_8$).
\end{table}

\begin{rem}
(a) It emerges that the distribution into modular Harish-Chandra series in
all examples considered only depends on the relative Weyl group.

(b) In addition to the first block of $\SO_{12}^+(q)$ and the first block of
$\SO_{16}^+(q)$, and the third  block of $\SO_{12}^+(q)$ and the fourth block
of $\SO_{16}^+(q)$, which form Morita equivalent pairs by
Theorem~\ref{thm:D8,d=4}, the following four pairs of blocks have identical
decomposition matrices (after suitably reordering the characters): the
principal block of $E_6$ and the 3rd block of $E_8$; the 2nd and the 3rd block
of $E_7$; the first blocks of $E_7$ and of $E_8$; the 4th blocks of $E_7$
and of $E_8$.
It is claimed (without proof) in \cite[Rem.~34]{Mi08} that the first listed
pair of blocks are in fact Morita equivalent. It would be interesting to see
whether this is true for all pairs mentioned above.

(c) The decomposition matrix for $\SO_8^+(q)$ and the one for the second block
of $\SO_{12}^+(q)$ have automorphisms induced by the non-trivial graph
automorphisms of the underlying groups. But note that also
the decomposition matrix for the principal block of $E_6$ has an
automorphism fixing $\phi_{6,1},\phi_{80,7},D_4\!:\!2.1$ and $\phi_{6,25}$ and
interchanging the other characters in pairs, and similarly, the decomposition
matrix for the second block of $E_7$ has an automorphism of order two with
fixed points $\phi_{280,8}$ and $\phi_{216,16}$.
\end{rem}

\section{Decomposition matrices for twisted type groups}   \label{sec:twist}

We now turned to simply-laced groups of twisted type, viz.~$\tw2D_5$,
$\tw2D_6$ and $\tw2E_6$.

\subsection{Decomposition matrices for $\SO_{10}^-(q)$}
The group $G=\SO_{10}^-(q)$ has four unipotent $\ell$-blocks for primes
$2\ne \ell|(q^2+1)$, the principal block, one block with cyclic defect and two
blocks of defect zero.

\begin{thm}   \label{thm:2D5,d=4}
 Assume that $q$ is odd. Then the $\ell$-modular decomposition matrices for
 the unipotent blocks of $\SO_{10}^-(q)$ of positive defect for primes $\ell$
 with $(q^2+1)_\ell>5$ are as given in Tables~\ref{tab:2D5,d=4}
 and~\ref{tab:2D5,d=4,def1}.
\end{thm}

\begin{table}[htbp]
\caption{$\SO_{10}^-(q)$, $q$ odd, $(q^2+1)_\ell>5$}   \label{tab:2D5,d=4}
$$\vbox{\offinterlineskip\halign{$#$\hfil\ \vrule height10pt depth4pt&&
      \hfil\ $#$\hfil\cr
    4.&                       1& 1\cr
   31.&            q\Ph3\Ph{10}& 1& 1\cr
  2^2.&          q^2\Ph8\Ph{10}& .& .& 1\cr
    .4& \hlf q^3\Ph6\Ph8\Ph{10}& 1& .& .& 1\cr
 21^2.& \hlf q^3\Ph3\Ph8\Ph{10}& .& 1& .& .& 1\cr
 21.1&\hlf q^3\Ph2^4\Ph6\Ph{10}& .& .& 1& .& .& 1\cr
 2.1^2&         q^6\Ph3\Ph6\Ph8& .& .& .& .& .& 1& 1\cr
 1^2.2&      q^5\Ph3\Ph8\Ph{10}& .& .& .& .& .& 1& .& 1\cr
  1^4.& \hlf q^7\Ph6\Ph8\Ph{10}& .& .& .& .& 1& .& .& .& 1\cr
   .31& \hlf q^7\Ph3\Ph8\Ph{10}& 1& 1& .& 1& .& .& .& .& .& 1\cr
 1.21&\hlf q^7\Ph2^4\Ph6\Ph{10}& .& .& 1& .& .& 1& 1& 1& .& .& 1\cr
  .2^2&       q^{10}\Ph8\Ph{10}& .& .& 1& .& .& .& .& .& .& .& 1& 1\cr
 .21^2&       q^{13}\Ph3\Ph{10}& .& 1& .& .& 1& .& .& .& .& 1& .& a& 1\cr
  .1^4&                  q^{20}& .& .& .& .& 1& .& .& .& 1& .& .& a\pl2& 1& 1\cr
\noalign{\hrule}
  & & ps& ps& ps& \tw2D_2& ps& ps& \tw2D_2& \tw2D_2& A_3& \tw2D_2& \tw2D_2& c& \tw2D_2& c\cr
  }}$$
  Here, $a \in \{0,1\}$.
\end{table}

\begin{table}[ht]
\caption{$\SO_{10}^-(q)$, block of defect 1, $2\ne\ell|(q^2+1)$} \label{tab:2D5,d=4,def1}
$$\vbox{\offinterlineskip\halign{$#$
        \vrule height10pt depth 2pt width 0pt&& \hfil$#$\hfil\cr
 3.1& \vr& 1.3\,& \vr& \bigcirc& \vr& 1.1^3& \vr& 1^3.1\cr
 & ps& & \tw2D_2& & \tw2D_2& & ps\cr
  }}$$
\end{table}


\begin{proof}
The Brauer tree for the block with cyclic defect is easily seen to be as given
in Table~\ref{tab:2D5,d=4,def1}. We describe how to obtain projectives
$\Psi_1,\ldots,\Psi_{11},\Psi_{13}, \Psi_{14}$ for the principal block as in
Table~\ref{tab:2D5,d=4}. \par
Application of (HCi) yields $\Psi_i$ with $i\in\{1,3,4,5,6,7,8,9,13\}$.
Moreover, we obtain $\Psi_2+\Psi_3$, $\Psi_2+\Psi_6$, so that (Sum) gives
$\Psi_2$. (Alternatively, the projectives in the principal series are obtained
from the Iwahori--Hecke algebra $H$ of type $B_4$, with parameters $q^2$ and
$q$.) Next, (HCi) gives $\tPsi_{10}=\Psi_{10}+\Psi_8$,
$\tPsi_{10}'=\Psi_{10}+\Psi_{11}$ and $\tPsi_{11}=\Psi_{11}+\Psi_{13}$. Thus,
$\tPsi_{10}+\tPsi_{11}=\tPsi_{10}'+\Psi_8+\Psi_{13}$, which shows that
$\Psi_8,\Psi_{13}$ occur as summands of $\tPsi_{10}+\tPsi_{11}$. Nonnegativity
of decomposition numbers implies that $\Psi_8$ occurs in $\tPsi_{10}$ and
$\Psi_{11}$ in $\tPsi_{11}$, so we obtain $\Psi_{10}$ and $\Psi_{11}$.
Application of (HCr) also shows that $\Psi_1,\ldots,\Psi_{11}$ and $\Psi_{13}$
are indecomposable. The last projective $\Psi_{14}$ is given by (St).
\par
The Hecke algebra for the cuspidal Brauer character of
$\tw2D_2(q)\cong A_1(q^2)$ is of type $B_3$, and the parameters are seen
locally in the Levi subgroups of types $\tw2D_4$ and $\tw2D_3\times A_1$ to
be $q^4$ and $q$. From its decomposition matrix it follows that exactly eight
simple modules lie in the corresponding Harish-Chandra series in $G$. Since the
unipotent block with cyclic defect contains two of them, the principal block
will contain the remaining six. We have then accounted for all non-cuspidal
Harish-Chandra series, so the remaining Brauer character must all be cuspidal.
\par
By (GGGR), the unipotent part of $\Psi_{12}$ is $\rho_{.2^2}+a \rho_{.21^2}
+b\rho_{.1^4}$, where $a$ and $b$ are unknown. To compute $b$ we use
(Cox), which gives the relation $b \leq a+2$. The relation $b \geq a+2$
is obtained from the $\ell$-reduction of the non-unipotent character
obtained by Deligne--Lusztig induction of an $\ell$-character in general
position of a torus of order $(q+1)(q^2+1)^2$ (such a character exists
whenever $(q^2+1)_\ell>5$). Finally, to obtain an upper bound for $a$, we
consider the generalized Gelfand--Graev representations of $\SO_{14}^-(q)$
associated to the family
$\{\rho_{1^5.1},\rho_{1^3.1^3},\rho_{1.2^21},\rho_{.321}\}$, and more precisely
the one whose projection to this family is $\rho_{.321}+\rho_{1.2^21}$.
The character of this representation, cut by the block containing
$\rho_{.321}$ is of the form $\rho_{.321}+ \rho_{1.2^21} + \alpha \rho_{1.1^5}$
by (GGGR). The Harish-Chandra restriction of this character yields $a \leq 1$.
\end{proof}

\begin{rem}
If Kawanaka's conjecture (see Remark~\ref{rem:kawanaka}) holds for the characters
in the family
$\{\rho_{1^5.1},\rho_{1^3.1^3},\rho_{1.2^21},\rho_{.321}\}$ of $\SO_{14}^-(q)$,
then the previous argument shows that $a=0$.
\end{rem}

\begin{rem}
The $5$-modular decomposition matrix of $\SO_{10}^-(2)$ is known; there,
the last three entries in the partially unknown 12th column of
Table~\ref{tab:2D5,d=4} read $(1,0,1)$, whence the case when $(q^2+1)_\ell=5$
does behave differently.
\end{rem}

\subsection{A decomposition matrix for $\SO_{14}^-(q)$}
The group $G=\SO_{14}^-(q)$ has one non-principal unipotent $\Phi_4$-block
of positive defect, which we label by its $4$-Harish-Chandra source in a Levi
subgroup of type $\SO_6^-(q)$.

\begin{thm}   \label{thm:2D7,d=4}
 Assume that $q$ is odd. Then the $\ell$-modular decomposition matrix for the
 non-principal unipotent block of $\SO_{14}^-(q)$ of defect $\Ph4^2$, for
 $(q^2+1)_\ell>5$, is as given in Table~\ref{tab:2D7,d=4,bl2}.
\end{thm}

\begin{table}[htbp]
\caption{$\SO_{14}^-(q)$, $q$ odd, block $\binom{0\,1\,3}{1}$, $(q^2+1)_\ell>5$}   \label{tab:2D7,d=4,bl2}
$$\vbox{\offinterlineskip\halign{$#$\hfil\ \vrule height10pt depth4pt&&
      \hfil\ $#$\hfil\cr
    5.1&                           \Ph3\Ph6& 1\cr
    1.5&     \hlf q^2\Ph3\Ph8\Ph{10}\Ph{14}& 1& 1\cr
   32.1&        \hlf q^2\Ph3\Ph5\Ph8\Ph{14}& 1& .& 1\cr
   321.&   \hlf q^2\Ph2^4\Ph6\Ph{10}\Ph{14}& .& .& .& 1\cr
 31^2.1&           q^4\Ph3^2\Ph6\Ph8\Ph{14}& .& .& .& 1& 1\cr
 2^21.1&      q^5\Ph3\Ph5\Ph6\Ph{10}\Ph{14}& .& .& 1& .& .& 1\cr
  1^3.3&          q^7\Ph5\Ph8\Ph{10}\Ph{14}& .& .& .& .& 1& .& 1\cr
   1.32&      q^9\Ph3\Ph5\Ph6\Ph{10}\Ph{14}& 1& 1& 1& .& .& .& .& 1\cr
  3.1^3&              q^{10}\Ph5\Ph8\Ph{10}& .& .& .& .& 1& .& .& .& 1\cr
 1.31^2&        q^{12}\Ph3^2\Ph6\Ph8\Ph{14}& .& .& .& 1& 1& .& 1& .& 1& 1\cr
   .321&\hlf q^{14}\Ph2^4\Ph6\Ph{10}\Ph{14}& .& .& .& 1& .& .& .& .& .& 1& 1\cr
 1.2^21&     \hlf q^{14}\Ph3\Ph5\Ph8\Ph{14}& .& .& 1& .& .& 1& .& 1& .& .& a& 1\cr
  1^5.1&  \hlf q^{14}\Ph3\Ph8\Ph{10}\Ph{14}& .& .& .& .& .& 1& .& .& .& .& .& .& 1\cr
  1.1^5&                     q^{28}\Ph3\Ph6& .& .& .& .& .& 1& .& .& .& .& a\pl2& 1& 1& 1\cr
\noalign{\hrule}
 \omit& & ps& \tw2D_2& ps& ps& ps& ps& \tw2D_2& \tw2D_2& \tw2D_2& \tw2D_2& .2^2& \tw2D_2& A_3& .1^4\cr
  }}$$
Here, $a \in \{0,1\}$ is as in Table~\ref{tab:2D5,d=4}, and all degrees have
been divided by $q^2\Ph4\Ph{12}$.
\end{table}

\begin{proof}
All columns but the 11th are obtained by (HCi), as well as
$\Psi_8+a\Psi_{10}+\Psi_{11}$ and $\Psi_{11}+(1-a)\Psi_{12}$. Thus,
independent of the value of $a$ we also recover $\Psi_{11}$ via (Sum).
\end{proof}

\subsection{Decomposition matrices for $\tw2E_6(q)$}
We now turn to the $\Phi_4$-blocks of the exceptional groups of type $\tw2E_6$.
There are 10 unipotent $\ell$-blocks of defect zero, one of cyclic defect and
the principal block.

\begin{thm}   \label{thm:2E6,d=4}
 Let $(q,6)=1$. The decomposition matrices for the unipotent $\ell$-blocks of
 $\tw2E_6(q)$ of positive defect, where $(q^2+1)_\ell>5$, are as given in
 Tables~\ref{tab:2E6,d=4} and~\ref{tab:2E6,d=4,def1}.
\end{thm}

\begin{table}[ht]
\caption{$\tw2E_6(q)$, $(q,6)=1$, $(q^2+1)_\ell>5$}   \label{tab:2E6,d=4}
$$\vbox{\offinterlineskip\halign{$#$\hfil\ \vrule height10pt depth4pt&&
      \hfil\ $#$\hfil\cr
  \phi_{1,0}&                                      1& 1\cr
 \phi_{2,4}'&                          q\Ph8\Ph{18}& .& 1\cr
  \phi_{9,2}&       \hlf q^3\Ph3^2\Ph8\Ph{10}\Ph{18}& 1& .& 1\cr
\phi_{1,12}'&     \hlf q^3\Ph8\Ph{10}\Ph{12}\Ph{18}& .& .& .& 1\cr
 \phi_{8,3}'&    \hlf q^3\Ph2^4\Ph6^2\Ph{10}\Ph{18}& .& 1& .& .& 1\cr
 \phi_{9,6}'&         q^6\Ph3^2\Ph6^3\Ph{12}\Ph{18}& .& .& .& 1& .& 1\cr
  \tw2E_6[1]&\sxt q^7\Ph1^4\Ph8\Ph{10}\Ph{12}\Ph{18}& .& .& .& .& .& .& 1\cr
\phi_{6,6}'&\thrd q^7\Ph3^2\Ph8\Ph{10}\Ph{12}\Ph{18}&.& .& .& .& 1& .&c_1& 1\cr
\phi_{6,6}''&\thrd q^7\Ph3^2\Ph6^3\Ph8\Ph{10}\Ph{12}& .& .& .& .& 1& .&  .& .& 1\cr
 \phi_{16,5}& \hlf q^7\Ph2^4\Ph6^2\Ph8\Ph{10}\Ph{18}& .& .& 1& .& .& 1&  .& .& .& 1\cr
 \phi_{9,6}''&      q^{10}\Ph3^2\Ph6^3\Ph{12}\Ph{18}& 1& .& 1& .& .& .&c_4& .& .& 1& 1\cr
\phi_{1,12}''&  \hlf q^{15}\Ph8\Ph{10}\Ph{12}\Ph{18}& 1& .& .& .& .& .&c_5& .& .& .& 1& 1\cr
 \phi_{9,10}&    \hlf q^{15}\Ph3^2\Ph8\Ph{10}\Ph{18}& .& .& .& 1& .& 1&c_6& .& .& 1& .&   .& 1\cr
 \phi_{8,9}''& \hlf q^{15}\Ph2^4\Ph6^2\Ph{10}\Ph{18}& .& 1& .& .& 1& .&c_7& 1& 1& .& .& d_2& .& 1\cr
 \phi_{2,16}''&                    q^{25}\Ph8\Ph{18}& .& 1& .& .& .& .&c_8& .& .& .& .& d_3& .& 1& 1\cr
 \phi_{1,24}&                                 q^{36}& .& .& .& 1& .& .&c_9& .& .& .& .& d_4&  1& .& 2& 1\cr
\noalign{\hrule}
 \omit& & ps& ps& ps& ps& ps& ps& \tw2E_6& \tw2D_2& \tw2D_2& \tw2D_2& \tw2D_2& c& \tw2D_2& \tw2D_2& c& c\cr
   }}$$
   Here, $c_1 \in \{0,1,2\}$, $d_2 \in \{0,1\}$, $c_9 = 4+2c_1+3c_4-3c_5+c_6-2c_7+2c_8$
   and $d_4 = -3-2d_2+2d_3$.
\end{table}

\begin{table}[ht]
\caption{$\tw2E_6(q)$, block of defect 1, $2\ne\ell|(q^2+1)$} \label{tab:2E6,d=4,def1}
$$\vbox{\offinterlineskip\halign{$#$
        \vrule height10pt depth 2pt width 0pt&& \hfil$#$\hfil\cr
 \phi_{4,1}& \vr& \phi_{4,7}''& \vr& \bigcirc& \vr& \phi_{4,13}& \vr& \phi_{4,7}'\cr
 & ps& & \tw2D_2& & \tw2D_2& & ps\cr
  }}$$
\end{table}

\begin{proof}
The Brauer tree for the block with cyclic defect is easily obtained, see also
\cite[Thm.~3.10]{Bre94}. We now discuss how to find projective characters
$\Psi_i$, for $i\in\{1$--$6,8$--$11,13,14,16\}$, with unipotent parts as given
in the columns of Table~\ref{tab:2E6,d=4}. (HCi) yields $\Psi_i$ with
$i\in\{1,3,4,6,11,13\}$, which are indecomposable by (HCr). Further,
we obtain $\Psi_2+\Psi_4$ and $\Psi_2+\Psi_3+\Psi_6$, yielding $\Psi_2$ by
(Sum). Similarly, $\Psi_5+2\Psi_6$ and $\Psi_5+\Psi_1+2\Psi_3$ lead to
$\Psi_5$; $\Psi_3+\Psi_8+\Psi_{11}$ and $\Psi_6+\Psi_8$ provide the projective
character $\Psi_8$, and $\Psi_3+\Psi_9$ and $\Psi_6+\Psi_9+\Psi_{13}$ lead
to $\Psi_9$. Furthermore, $\Psi_8+\Psi_{10}$ and $\Psi_9+\Psi_{10}$ give
$\Psi_{10}$, and $\Psi_{10}+\Psi_{14}$ and $\Psi_{11}+\Psi_{14}$ give rise
to $\Psi_{14}$. By inspection using (HCr) all projectives constructed so far
are indecomposable.   \par
We claim that we have now accounted for all non-cuspidal Harish-Chandra
series. Indeed, the decomposition numbers for the Hecke algebra of type $F_4$
with unequal parameters have been calculated in \cite[Thm.~3.10]{Bre94} for
all $\ell\ge5$, showing that there are six principal series PIMs. The
relative Weyl group of the cuspidal unipotent Brauer character of
$\tw2D_2(q)=A_1(q^2)$ has type $B_3$. We have already found all projective
indecomposable summands in the principal block of the Harish-Chandra induction
from proper Levi subgroups in that series, viz.~$\Psi_8,\Psi_9,\Psi_{10},
\Psi_{11},\Psi_{13}$ and $\Psi_{14}$. Two further PIMs in that series lie in
the block of cyclic defect. This accounts for all non-cuspidal Harish-Chandra
series. Hence the four missing columns must correspond to cuspidal Brauer
characters.
\par
For the remaining columns we consider the following three GGGRs, whose
existence is given by \cite[Thm.~6.5(ii)]{DLM14}, assuming that $p$ is good:
\begin{itemize}
 \item the GGGR associated to the family containing $\tw2E_6[1]$
  and with projection $\tw2E_6[1] + 2\phi_{6,6}'+\phi_{12,4}$ on this family;
 \item the GGGR associated to the family $\{\phi_{2,16}',\phi_{9,10},
  \phi_{1,12}'',\phi_{8,9}''\}$  and with projection
  $\phi_{1,12}''+\phi_{8,9}''$ on this family;
 \item the GGGR associated to the family $\{\phi_{2,16}''\}$.
\end{itemize}
>From (GGGR) we deduce the unitriangularity of the decomposition matrix.
In addition, if we denote by $c_1,\ldots, c_9$ the unknown entries in the $7$th
column then $c_2=c_3 = 0$ and $c_1 \leq 2$. Similarly, if $d_1,\ldots,d_4$
denote the entries in the $12$th column then $d_1 = 0$ and
$d_2 \leq 1$. The last unknown entry (in the $15$th column) will be denoted
by $d_5$.
\par
A Sylow $\Phi_4$-torus of $G$ has a regular $\ell$-character $\theta$ whenever
$(q^2+1)_\ell>5$. By (Red), the $\ell$-reduction of a non-unipotent character
induced from $\theta$ yields relations on the $c_i$'s and $d_i$'s, namely
$d_5 \geq 2$, $3+2d_2-2d_3+d_4 \geq 0$ and
$-4-2c_1-3c_4+3c_5-c_6+2c_7-2c_8+c_9 \geq 0$. To obtain the opposite
inequalities
we use (DL) successively for the elements $s_1 s_2 s_3 s_4$,
$s_1s_2s_3s_1s_4s_3$ and $s_1s_2s_4s_3s_1s_5s_4s_3s_6s_5s_4s_3$.
\end{proof}

As in Proposition \ref{prop:conjForD7}, we can use \cite[Conj.~1.2]{DM14} to
obtained conjectural upper bounds on the unknown entries in the decomposition
matrix which do not depend on $q$.

\begin{prop}
 Assume Conjecture~1.2 in \cite{DM14} holds. Then in the decomposition matrix
 of the principal $\Phi_4$-block of $\tw2E_6(q)$, we have $c_1 = 0$,
 $c_4 \leq 3$,
 $c_5 \leq 26$, $c_6 \leq 29$, $c_7 \leq 50$, $c_8 \leq 156$ and $d_3 \leq 6$. 
\end{prop}

\begin{proof}
We consider virtual characters $Q_w$ afforded by the Alvis--Curtis dual of
the intersection cohomology of suitably chosen Deligne--Lusztig varieties.
In the following table, we give, for each element $w$ we consider, 
the multiplicity of a PIM $\Psi_i$ in $Q_w[\lambda]$. In order to simplify
notation, we denote $s_1,\ldots,s_6$ by $1,\ldots,6$. 
$$\begin{array}{r|c|c|l}
w & \lambda & i & \langle Q_w[\lambda], \varphi_i \rangle \\[4pt]\hline
\vphantom{\Big)}
1231454236542314356	& 1	& 8	& -36c_1 \\	
123423145431		& 1	& 11	& 3-c_4\\
			&	& 12	& 23+c_4-c_5\\[4pt]
2342314354316543	& 1	& 13	& 6(29-c_6) \\ 
			&	& 14	& 6(50+c_1-d_2(23+c_4-c_5)-c_7)\\[4pt]
23143154316543	& 1	& 15	& 2(106-c_1+c_7+(d_2-d_3)(23+c_4-c_5)-c_8)\\[4pt]
546542		& -1	& 15	& 5+d_2-d_3 \\[4pt]
\end{array}$$
Conjecture~1.2 in \cite{DM14} predicts that the entries in the last column
of the previous table are non-negative. We deduce that
$c_1 = 0$, $c_4 \leq 3$, $c_5 \leq c_4+23 \leq 26$, $c_6 \leq 29$,
$c_7 \leq 50+c_1-d_2(23+c_4-c_5) \leq 50$, 
$$c_8 \leq 106-c_1+c_7+(d_2-d_3)(23+c_4-c_5) \leq 106-c_1+c_7+d_2(23+c_4-c_5)
\leq 156$$
and $d_3 \leq d_2 + 5 \leq 6$ since $d_2 \in \{0,1\}$.
\end{proof}

\subsection{The $\Phi_4$-blocks in twisted groups}
As in the untwisted case we collect some data on the $\Phi_4$-blocks studied
above in Table~\ref{tab:4-blocks, twisted}.

\begin{table}[ht]
\caption{HC-series in $\Phi_4$-blocks of twisted groups}    \label{tab:4-blocks, twisted}
$$\begin{array}{r|cc|ccccccccc}
 G&  W_G(b)& |\IBr(b)|& ps& \tw2D_2& A_3& .2^2& .1^4& c\\
\hline
\tw2D_5&    G(4,1,2)& 14&  5&  6& 1&  1& 1&  \\
\tw2D_7&            &   &  5&  6& 1&  1& 1&  \\
\hline
\tw2E_6&         G_8& 16&  6&  6&  &   &  & 4\\
\end{array}$$
\end{table}

Note that the decomposition matrices for the blocks in $\tw2D_5$ and in
$\tw2D_7$ of defect $\Phi_4^2$ coincide after permuting rows and columns
suitably. Again, it would be interesting to see whether this is caused by a
Morita equivalence between these blocks. On the other hand, the multisets of
entries of the matrices for twisted groups differ from those for any of the
untwisted ones, so if there exits a Morita equivalence between blocks
for twisted and untwisted groups, it would have to be with respect to a
different choice of basic sets.

\section{Decomposition matrices for symplectic groups and $F_4(q)$\label{sec:F4}}

We now turn to groups with non-simply laced Dynkin diagram, where we start by
giving (approximations) to decomposition matrices for the unipotent blocks of
small rank symplectic groups $\Sp_{2n}(q)$ for primes $\ell|(q^2+1)$.
Again, it is not known a priori in our present situation that the
decomposition matrix has triangular shape. This leads to additional
complications.

For completeness and for use in the subsequent proofs, we recall the known
Brauer trees for $\Sp_4(q)$ and $\Sp_6(q)$ (see \cite{FS2}). We
also indicate the modular Harish-Chandra series of the PIMs.

\begin{thm}[Fong--Srinivasan]  \label{thm:C2}
 Let $2\ne\ell|(q^2+1)$ be a prime. Then the Brauer trees for the unipotent
  $\ell$-blocks of $\Sp_4(q)$ and $\Sp_6(q)$ are as given in
  Table~\ref{tab:C2u3,d=4}.
\end{thm}

\begin{table}[ht]
\caption{$\Sp_4(q)$ and $\Sp_6(q)$, $2\ne\ell|(q^2+1)$}  \label{tab:C2u3,d=4}
$$\begin{array}{cccccccccccc}
\Sp_4(q):\qquad& 2.& \vr& 1.1& \vr& .1^2& \vr& \bigcirc& \vr& C_2\cr
 & & ps& & ps& & c& & C_2\cr
 &   & & \\
\Sp_6(q):\qquad& 3.& \vr& 1.2& \vr& .21& \vr& \bigcirc& \vr& C_2:1^2\cr
\cr
 & 21.& \vr& 1^2.1& \vr& .1^3& \vr& \bigcirc& \vr& C_2:2\cr
 & & ps& & ps& & .1^2& & C_2\cr
\end{array}$$
\end{table}

Next, let $G=\Sp_8(q)$.

\begin{thm}   \label{thm:C4}
 Let $\ell$ be a prime. Assume that $q$ is odd. Then the decomposition
 matrices for the unipotent $\ell-$ blocks of $\Sp_8(q)$, $(q^2+1)_\ell>5$,
 are as given in Tables~\ref{tab:C4,d=4} and~\ref{tab:C4,d=4,def1}.
\end{thm}

\begin{table}[ht]
\caption{$\Sp_8(q)$, $(q^2+1)_\ell>5$}   \label{tab:C4,d=4}
$$\vbox{\offinterlineskip\halign{$#$\hfil\ \vrule height10pt depth4pt&&
      \hfil\ $#$\hfil\cr
      4.&                           1& 1\cr
      .4&              \hlf q\Ph6\Ph8& .& 1\cr
     31.&              \hlf q\Ph3\Ph8& 1& .& 1\cr
     1.3&      \hlf q^2\Ph2^2\Ph6\Ph8& 1& 1& .& 1\cr
C_2:1^2.&      \hlf q^2\Ph1^2\Ph3\Ph8& .& .& .& .& 1\cr
   21^2.&        \hlf q^4\Ph3\Ph6\Ph8& .& .& 1& .& .& 1\cr
     .31&        \hlf q^4\Ph3\Ph6\Ph8& .& 1& .& 1& .& .& 1\cr
 C_2:1.1&\hlf q^4\Ph1^2\Ph2^2\Ph3\Ph6& .& .& .& .& 1& .& .& 1\cr
   1^2.2&             q^4\Ph3\Ph6\Ph8& 1& .& 1& 1& .& .& .& .& 1\cr
  C_2:.2&      \hlf q^6\Ph1^2\Ph3\Ph8& .& .& .& .& .& .& .& 1& .& 1\cr
   1^3.1&      \hlf q^6\Ph2^2\Ph6\Ph8& .& .& 1& .& .& 1& .& .& 1& a& 1\cr
    1^4.&            \hlf q^9\Ph6\Ph8& .& .& .& .& .& 1& .& .& .& a& 1& 1\cr
   .21^2&            \hlf q^9\Ph3\Ph8& .& .& .& 1& .& .& 1& .& 1& .& .& .& 1\cr
    .1^4&                      q^{16}& .& .& .& .& .& .& .& .& 1&a\pl2& 1& b& 1& 1\cr
\noalign{\hrule}
  & & ps& ps& ps& ps& C_2& ps& .1^2& C_2& ps& c& A_3& c& .1^2& c\cr
  }}$$
  Here, $a \in \{0,1\}$ and $b \in \{0,1,2,3\}$.
\end{table}

\begin{table}[ht]
\caption{$\Sp_8(q)$, blocks of defect 1, $2\ne\ell|(q^2+1)$} \label{tab:C4,d=4,def1}
$$\vbox{\offinterlineskip\halign{$#$
        \vrule height10pt depth 2pt width 0pt&& \hfil$#$\hfil\cr
 3.1& \vr& 2.2& \vr& .2^2& \vr& \bigcirc& \vr& C_2\!:\!.1^2\cr
 \cr
 2^2.& \vr& 1^2.1^2& \vr& 1.1^3& \vr& \bigcirc& \vr& C_2\!:\!2.\cr
 & ps& & ps& & .1^2& & C_2\cr
  }}$$
\end{table}

\begin{proof}
The group $G=\Sp_8(q)$ has three unipotent blocks of $\Phi_4$-defect zero, two
blocks of $\Phi_4$-defect~1 with four characters each, and all other unipotent
characters lie in the principal block. The Brauer trees for the blocks of
defect~1 are known by \cite{FS2} and in any case can easily be recovered by
(HCi).
\par
The projective modules $\Psi_i$ for $i \in \{1,..,9,11,13\}$ are obtained
by (HCi). Using (HCr), we can check that there are indecomposable. Finally,
(St) yields the last column, leaving only two (necessarily cuspidal)
PIMs to determine. A first approximation of these columns is given by (GGGR)
for the families $\{\rho_{1^3.1},\rho_{1^2.1^2},\rho_{.2^2},\rho_{C_2:.2}\}$ and
$\{\rho_{1^4.},\rho_{1.1^3},\rho_{.21^2},\rho_{C_2:.1^2}\}$.
We deduce that the relevant submatrix for the last five projectives
now has the form
$$\vbox{\offinterlineskip\halign{$#$\hfil\ \vrule height10pt depth4pt&&
      \hfil\ $#$\hfil\cr
C_2:.2&     1\cr
 1^3.1&   a_1& 1\cr
  1^4.&   a_2& 1& 1\cr
 .21^2&   a_3& .& a_5& 1\cr
  .1^4&   a_4& 1& a_6& 1& 1\cr
  }}$$
In addition, \cite[Thm.~6.5(ii)]{DLM14} yields $a_1 \in\{0,1\}$ and $a_5 =0$.
\par
As usual, relations on the $a_i$'s are obtained by looking at suitable
Deligne--Lusztig characters: from (Cox) we deduce that $a_1-a_2 \geq 0$ and
$a_1+a_3-a_4 + 2 \geq a_6(a_1-a_2)$. But from (Red) with the induction of an
$\ell$-character of a Sylow $\Phi_4$-torus in general position, which exists
whenever $(q^2+1)_\ell>5$, we get $a_1+a_3-a_4 +2  \leq a_2-a_1$. Consequently,
$0 \leq a_6(a_1-a_2) \leq a_2-a_1\leq 0$ which forces $a_1 = a_2$ and
$a_4 = a_1+a_3+2$. With (DL) applied to the element
$w = s_1s_2s_3s_2s_1s_2s_3s_4$ of the Weyl group we obtain $a_6 \leq 3$.
\par
Finally, we use the GGGR of $\Sp_{10}(q)$ associated to the family
$\{\rho_{1^4.1},\rho_{1^2.1^3},\rho_{.2^21},\rho_{C_2:.21}\}$ with projection
$\rho_{C_2:.21}+\rho_{1^4.1}$ to this family. Cut by the block containing
these two characters, the unipotent part of this GGGR is of the form
$\rho_{C_2:.21}+\rho_{1^4.1}+\alpha\rho_{.1^5}$. Its Harish-Chandra restriction
to $G$ forces $a_3 = 0$. Setting $a=a_1$ and $b=a_6$,
we obtain the decomposition matrix as shown in Table~\ref{tab:C4,d=4}
\end{proof}

\begin{rem}
Kawanaka's conjecture \cite[Conj.~2.4.5]{Ka87} would imply $a=0$ (see also
Remark \ref{rem:kawanaka}).
\end{rem}

\begin{thm}   \label{thm:F4,d=4}
 Let $\ell$ be a prime. The decomposition matrix for the principal $\ell$-block
 of $F_4(q)$, $(q,6)=1$, $(q^2+1)_\ell>5$, is as given in Table~\ref{tab:F4,d=4}.
\end{thm}

For the unipotent blocks with cyclic defect, see \cite[Lemma~5.4]{Koe06}.

\begin{table}[ht]
\caption{$F_4(q)$, $(q,6)=1$, $(q^2+1)_\ell>5$}   \label{tab:F4,d=4}
$$\vbox{\offinterlineskip\halign{$#$\hfil\ \vrule height10pt depth4pt&&
      \hfil\ $#$\hfil\cr
 \phi_{1,0}&                                      1& 1\cr
 \phi_{4,1}&                 \hlf q\Ph2^2\Ph6^2\Ph8& .& 1\cr
     B_2:2.&                 \hlf q\Ph1^2\Ph3^2\Ph8& .& .& 1\cr
 \phi_{9,2}&                 q^2\Ph3^2\Ph6^2\Ph{12}& 1& 1& .& 1\cr
\phi_{12,4}& \frac{1}{24}q^4\Ph2^4\Ph3^2\Ph8\Ph{12}& .& 1& .& 1& 1\cr
 \phi_{4,8}&  \frac{1}{8}q^4\Ph2^4\Ph6^2\Ph8\Ph{12}& .& .& .& 1& .& 1\cr
\phi_{6,6}''&      \thrd q^4\Ph3^2\Ph6^2\Ph8\Ph{12}& 1& .& .& 1& .& .& 1\cr
 B_2:1.1&\frac{1}{4}q^4\Ph1^2\Ph2^2\Ph3^2\Ph6^2\Ph8& .& .& 1& .& .& .& .& 1\cr
F_4^{II}[1]& \frac{1}{24}q^4\Ph1^4\Ph6^2\Ph8\Ph{12}& .& .& .& .& .& .& .& .& 1\cr
   F_4^I[1]&  \frac{1}{8}q^4\Ph1^4\Ph3^2\Ph8\Ph{12}& .& .& .& .& .& .& .& .& .& 1\cr
    F_4[-I]& \frac{1}{4}q^4\Ph1^4\Ph2^4\Ph3^2\Ph6^2& .& .& .& .& .& .& .& .& .& .& 1\cr
     F_4[I]& \frac{1}{4}q^4\Ph1^4\Ph2^4\Ph3^2\Ph6^2& .& .& .& .& .& .& .& .& .& .& .& 1\cr
\phi_{9,10}&              q^{10}\Ph3^2\Ph6^2\Ph{12}& .& .& .& 1& 1& 1& 1& .& .& .& c_1& c_1& 1\cr
\phi_{4,13}&            \hlf q^{13}\Ph2^2\Ph6^2\Ph8& .& .& .& .& 1& .& .& .& a_1& .& c_1\mn1& c_1\mn1& 1& 1\cr
   B_2:.1^2&            \hlf q^{13}\Ph1^2\Ph3^2\Ph8& .& .& .& .& .& .& .& 1& .& b_1& c_3& c_3& .& .& 1\cr
\phi_{1,24}&                                   q^24& .& .& .& .& .& .& 1& .& a_2& 2b_1\mn3& c_4& c_4& 1& d& 2& 1\cr
\noalign{\hrule}
  & & ps& ps& B_2& ps& ps& .1^2& ps& B_2& c& c& c& c& .1^2& c& c& c\cr
  }}$$
  Here, $a_1 \leq 5$, $a_2 \leq 13+(5-a_1)d$, $c_3 \in \{0,1\}$, $c_4 = c_1+2c_3-2$  and $d \in \{0,1,2\}$.
\end{table}

\begin{proof}
We start from the approximation to the decomposition matrix which was obtained
in the thesis of K\"ohler \cite{Koe06}. The relevant submatrix for the last
eight projectives has the following form:
$$\begin{array}{l|cccccccc}
F_4^{II}[1]&    1\cr
   F_4^I[1]&   .&   1\cr
    F_4[-I]&   .&   .&    1\cr
     F_4[I]&   .&   .&   .& 1\cr
\phi_{9,10}&   .&   .& c_1& c_1& 1\cr
\phi_{4,13}& a_1&   .& c_2& c_2& 1& 1\cr
   B_2:.1^2&   .& b_1& c_3& c_3& .& .& 1\cr
\phi_{1,24}& a_2& b_2& c_4& c_4& 1& d& e& 1\cr
\end{array}$$
First relations come from the $\ell$-reduction of non-unipotent characters.
For each uniform character $\rho\in\cE(G,(s))$ we construct,
we give in the table below the type of $C_\bG(s)$, the Jordan correspondent
$\rho_s$ of $\rho$ and the relations that we will use:
$$\begin{array}{c|c|c} C_\bG(s) & \rho_s  & \text{relations} \\\hline
B_2 \cdot (q^2+1) & \rho_{1^2.}+\rho_{B_2} &  c_2 \geq c_1-1 \\\hline
(q^2+1)^2 & 1 &  b_2 \geq 2b_1-3 \\
			& & c_4 \geq 3c_1-2c_2+2c_3-4 \\
			& & e\geq 2 \\
\end{array}$$
Now we apply (DL) successively to obtain relations on the unknown entries.
Starting with the Deligne--Lusztig character associated with a Coxeter element
$w$ we find the non-negative coefficients $c_1-c_2-1$, $3-2c_3$ and
$2+2c_1-2c_4-3e+2c_3e$ in $R_w$ of the PIMs corresponding to columns
9,10,14~and~16,
so that $c_2 = c_1-1$ and $c_3\in \{0,1\}$. The second relation can then be
written $-2(4-3c_1+2c_2-2c_3+c_4)+(3-2c_3)(2-e) \geq 0$. Since it is a sum of
two nonpositive integers, we obtain $4-3c_1+2c_2-2c_3+c_4 =0$
and $(3-2c_3)(2-e) = 0$, so that $c_4 = c_1+2c_3-2$ and $e=2$. Consequently,
the PIMs corresponding to the columns $9$, $10$, $14$ and $16$ do not occur
in $R_w$.
\par
We continue with the Deligne--Lusztig character associated with
$w' =  s_1s_2s_3s_4s_1s_2s_3s_4$. Using (DL) we find $2b_1-b_2-3 \geq 0$
which forces $b_2 = 2b_1-3$ by the previous inequalities. In addition,
the PIMs corresponding to the columns $9$, $14$ and $16$ still do not occur
in $R_{w'}$.
\par
Finally, with $w'' = s_1s_2s_3s_4s_1s_2s_3s_4s_1s_2s_3s_4$  we consider
the characters $R_{w''}[\lambda]$ for various eigenvalues $\lambda$ of $F$.
The multiplicities of the $14$th and $16$th PIM in these virtual
characters yield the relations
$$a_1 \leq 5,\quad d \leq 2,\quad a_2 \leq 13+ (a_1-5) d$$
(in particular $a_2 \leq 13$).
\end{proof}

Up to finitely many possibilities, the decomposition matrix given above depends
only on two unknown parameters, viz.~$b_1,c_1$. Moreover, these could be
bounded above by suitable polynomials in $q$ using GGGRs. As in
Proposition~\ref{prop:conjForD7}, we can also produce conjectural bounds
independent of $q$.

\begin{prop}
 Assume Conjecture~1.2 in \cite{DM14} holds. Then in the decomposition matrix
 of the principal $\Phi_4$-block of $F_4(q)$ we have $b_1 \in \{2,3,4\}$ and
 $c_1 \in \{1,2\}$.
\end{prop}

\begin{proof}
We compute the Alvis--Curtis dual of the intersection cohomology of two
Deligne--Lusztig varieties, corresponding to the elements
$w_1 = s_1 s_2 s_3 s_4$ and $w_2 =s_2s_4s_3s_2s_1s_3s_2s_3 $. Conjecture~1.2 in \cite{DM14}
predicts that the corresponding characters, denoted in \cite{DM14} by $Q_{w_1}$ and $Q_{w_2}$,
are, up to sign, the unipotent part of projective characters. The same holds for
the generalized eigenspaces of $F$ on these characters. The multiplicity of $\Psi_{13}$
in $Q_{w_1}$ is $5-2c_1$, which forces $c_1 \leq 2$, and the multiplicity of $\Psi_{15}$ in
the $1$-eigenspace of $F$ on $Q_{w_2}$ is $4-b_1$, which forces $b_1 \leq 4$.
\end{proof}

We collect information on the Harish-Chandra series for the blocks considered
in this section in the subsequent Table~\ref{tab:4-blocks2}.

\begin{table}[ht]
\caption{HC-series in $\Phi_4$-blocks of defect $\Phi_4^2$}    \label{tab:4-blocks2}
$$\begin{array}{r|cc|ccccccc}
 G& W_G(b)& |\IBr(b)|& ps& B_2& .1^2& C_2& .1^2& A_3& c\\
\hline
    B_4   & G(4,1,2)& 14& 6& 2& 2&  &  & 1& 3\\
    C_4   &         &   & 6&  &  & 2& 2& 1& 3\\
\hline
    F_4   &      G_8& 16& 5& 1& 1& 1& 1&  & 7\\
\end{array}$$
\end{table}


\end{document}